\documentclass{article}
\usepackage[utf8]{inputenc}
\usepackage{graphicx} 
\usepackage{hyperref}
\usepackage{amsmath}
\usepackage[utf8]{inputenc}
\usepackage[english]{babel}
\usepackage{xcolor}
\usepackage[margin=1.1in]{geometry}
\usepackage{enumerate}
\usepackage{amssymb}
\usepackage{amsthm}
\usepackage{mathrsfs}
\usepackage{lipsum} 
\usepackage{mathtools}
\usepackage{float}
\usepackage{braket}
\usepackage{cancel}
\usepackage{eufrak}
\usepackage{comment}
\usepackage{algorithm}
\usepackage{algorithmic}
\usepackage{float}
\usepackage{textcomp}

\newtheorem{prop}{Proposition}[section]
\newtheorem{theorem}{Theorem}[section]

\newtheorem{lemma}[theorem]{Lemma}
\newtheorem{definition}{Definition}[section]
\usepackage[normalem]{ulem}
\usepackage{titlesec}
\usepackage{natbib}
\usepackage{cleveref}
\usepackage{wrapfig}

\titleformat{\paragraph}
{\normalfont\normalsize\bfseries}{\theparagraph}{1em}{}
\titlespacing*{\paragraph}
{0pt}{3.25ex plus 1ex minus .2ex}{1.5ex plus .2ex}

\begin{document}

\title{More Vertices of the Tristochastic Polytope}

\author{{Nati Linial\thanks{School of Computer Science and Engineering, Hebrew University, Jerusalem 91904, Israel. e-mail: nati@cs.huji.ac.il.{~Supported in part by an ERC Grant 101141253, "Packing in Discrete Domains - Geometry and Analysis".}}}\and{Zur Luria\thanks{School of Software Engineering and Computer Science, Azrieli College of Engineering, Jerusalem, Israel e-mail: zurlu@jce.ac.il}}\and Maya Trakhtman\thanks{M.sc. student at the School of Computer Science and Engineering, Hebrew University, Jerusalem 91904, Israel. e-mail: maya.trakhtman@mail.huji.ac.il.}}
\date{}
\maketitle
\begin{abstract}
The $n\times n$ doubly 
stochastic matrices constitute a polytope in $\mathbb{R}^{n^2}$, and by
Birkhoff's theorem, its vertex set coincides
with the set of order-$n$ permutation matrices.\\
A tristochastic array is an $n \times n\times n$ array
of nonnegative reals, where each row, column, and shaft sums to one.
These arrays constitute a polytope $\Delta_n$ in $\mathbb{R}^{n^3}$. 
In analogy, it is easy to see that each of the $L_n$
order-$n$ Latin squares is a vertex of $\Delta_n$,
but in contrast to Birkhoff's theorem, 
Latin squares form a vanishingly small subset of $\Delta_n$'s vertex set.
We show here that $\Delta_n$ has at least
$L_n^{2-o(1)}$ vertices. 
\end{abstract}

\section{Introduction}

A {\em doubly stochastic} matrix is a real  
$n \times n$ matrix with nonnegative entries,
in which every row and column sums to one. The set of order-$n$ doubly 
stochastic matrices forms a polytope 
$\Omega_n\subset\mathbb{R}^{n^2}$, and every order-$n$ 
permutation matrix is clearly a vertex of $\Omega_n$. A classical theorem of
Birkhoff \cite{birkhoff1946tres} states that
these are in fact $\Omega_n$'s only vertices.

Moving up in dimensions, 
an $n \times n\times n$ array is
\textit{tristochastic} if its entries are nonnegative reals, 
and each row, column, and shaft sums to one.
The set of order-$n$ tristochastic arrays forms a polytope 
$\Delta_n \subset \mathbb{R}^{n^3}$. 
It is easy to see that a tristochastic array of only $0,1$ 
entries\footnote{The view of Latin squares as two-dimensional permutations
emerged in a body of literature that is
called {\it high-dimensional combinatorics}, e.g., \cite{luria_count_hd_prmt}.}
is synonymous with an order-$n$ {\em Latin square}, and that every Latin square 
is a vertex of $\Delta_n$. However, in contrast to Birkhoff's theorem, 
Latin squares form a vanishingly small subset of $\Delta_n$'s vertex set.
It is known that the number of order-$n$ Latin squares is 
$((1+o(1))\frac{n}{e^2})^{n^2}$ (see \cite{van_lint_wilson}).  
In contrast, as shown in \cite{linial2014vertices},
$\Delta_n$ has at least $n^{(\frac{3}{2} - o(1)) n^2}$ vertices. 

In this paper we improve this lower bound, and show that $\Delta_n$ has at least
$n^{(2 - o(1)) n^2}$ vertices.

\subsection{Definitions and notations} 

Let $M$ be a $3$-dimensional array with real entries.
\begin{itemize}
\item The set of entries of the form $M(\cdot,j,k)$, $M(i, \cdot, k)$, and 
$M(i,j,\cdot)$ are called the $(j,k)$-{\em column}, 
the $(i,k)$-{\em row}, and the $(i,j)$-{\em shaft}, of $M$ respectively. 
\item A {\em line} in $M$ refers to either a column, a row or a shaft.
\item The $t$-th {\em $x$-wall} (resp.\  the $t$-th {\em $y$-wall}) of
$M$ is the set of its entries of the form $M(t,\cdot,\cdot)$, 
(resp.\ $M(\cdot,t,\cdot)$). 
\item The $t$-th {\em $z$-wall} is also called the $t$-th {\em layer}. It is the set of the 
entries of the form $M(\cdot,\cdot,t)$, denoted by $M^{(t)}$. 
\item The $(a,b)$-th {\em slice} of $M$, denoted by $M^{(a,b)}$ is the union of layers $a,\ldots,b$ of $M$. 
\item Associated with every slice $M^{(1,k)}$ is a matrix $D^k=D(M^{(1,k)})$, whose entries are the shaft sums up to and including layer $k$. We think of $D$ as providing a view of $M^{(1,k)}$ from above. 
\end{itemize}
The support $\text{supp}(R)$ of an array $R$ is the set of its non-zero cells.

\begin{definition}
Let $\cal H= \cal H\rm_n$ be the set of all tristochastic arrays in 
$\Delta_n$ all of whose entries are either $0$ or $\frac12$. In other words,
every line has two $\frac12$-entries and the rest are zeros. 
\end{definition}

Throughout this paper, we construct arrays $H\in \cal H$ layer by layer, from the bottom up. In each layer we place $\frac{1}{2}$'s in selected cells of the evolving array.
Consequently, shaft sums of the evolving array $H$, and equivalently the entries in $D(H)$, take the values $0, \frac{1}{2}$ or $1$. 
We refer to these as {\em$0$-cells}, {\em$\frac{1}{2}$-cells}, and {\em$1$-cells}
of $D$ respectively.

We associate several graphs with the evolving $H$.  
\begin{definition}
Let $R$ be a $3D$ array,
every line of which has support size
at most $2$. The associated graph $U_R$ has  $\text{supp}(R)$ as its vertex set.
Adjacency means that the corresponding cells reside on the same line in $R$.
When $R$ comprises the lower $\ell$
layers in the evolving $H$ we denote it by $U_{\ell}$ rather than $U_R$.
\end{definition}

\begin{definition}
Associated with $H$ are also
the bipartite graphs $G$
$=\langle X\cup Y, E\rangle$ and 
$\overline{G}=\langle X\cup Y, \overline{E}\rangle$. Here
$X$ and $Y$ are the sets of $D$'s rows resp.\ columns, whereas
$E:=\{(x,y)\in X\times Y~|~D(x,y)=0\}$, and 
$\overline{E}:=\{(x,y)\in X\times Y~|~D(x,y)=\frac{1}{2}\}$.
As above, $D$ is the matrix
where $D(i,j)$ equals the sum of entries in the $(i,j)$-shaft
in the evolving array $H$.
\end{definition}

Lastly, recall the following standard facts.
\begin{prop}\label[proposition]{prop:bi-colored connected graph}
\begin{itemize}
\item[] 
\item
A graph is bipartite iff it contains no odd cycle.
\item 
A graph is connected and bipartite iff it admits a unique $2$-coloring (up to switching colors).
\end{itemize}
\end{prop}

\section{The construction}\label{section:2-the-construction}
Our goal is to construct many arrays $H\in\cal H$, for which $U_H$ 
is connected and not bipartite. As \cref{lemma:who_is_vertex} shows, such $H$
is a vertex of $\Delta_n$.
We construct $H$ layer by layer, and
follow the evolution of the graph $U_H$ as $H$ is being constructed. 

Here is the condition under which $H\in \cal H$ is a vertex of $\Delta_n$. 

\begin{lemma}\label[lemma]{lemma:who_is_vertex}
An array $H\in\cal H$ is a vertex of $\Delta_n$ iff no connected component of $U_H$ is bipartite. 
\end{lemma}


\begin{proof}
For $H\in\cal H$ that is not a vertex of $\Delta_n$,
we find a bipartite connected component of $U_H$. Since $H$
is not a vertex of $\Delta_n$, it is a convex combination of two members
$A\neq B\in \Delta_n$, viz., $H=\alpha \cdot A + (1-\alpha)\cdot B$, where $0<\alpha<1$. 
Consequently, $\text{supp}(H)=\text{supp}(A)\cup\text{supp}(B)$. 
Let's consider a cell where $A\ne B$ differ, 
i.e., $A(i_0,j_0,k_0)=a$, $B(i_0,j_0,k_0)=b$, with $a\ne b$.
Since $\alpha \cdot a + (1-\alpha)\cdot b = H(i_0,j_0,k_0)= \frac{1}{2}$,
the number $\frac{1}{2}$ is a strict convex combination of the
numbers $a$ and $b$, and in particular, $a, b\neq \frac{1}{2}$. 

Let $u_0\in V(U_H)$ correspond to the cell $H(i_0,j_0,k_0)$, and let
$u_1$ be a neighbor of $u_0$ in $U_H$, that corresponds to some cell $H(i_1,j_1,k_1)$. 
Since $A(i_0,j_0,k_0)=a$ and $\text{supp}(A)\subseteq\text{supp}(H)$, the definition of 
$\Delta_n$ implies that $A(i_1,j_1,k_1)=1-a$. Likewise, any neighbor of $u_1$ in $U_H$
corresponds to a cell that takes the value $a$ in $A$, and so on.
Note that $a\ne 1-a$ (as $a\ne \frac{1}{2}$). Consequently, the connected component of 
$u_0$ in $U_H$ is (uniquely) $2$-colored. Namely, the vertices whose corresponding cells in $A$
equal $a$ vs.\ those whose corresponding cell equals $1-a$. 
By \cref{prop:bi-colored connected graph} the connected component is bipartite.
 


Now assume that $U_H$ contains a bipartite connected component, with vertex bi-partition $(T,R)$. We construct two corresponding arrays $A$ and $B$ as follows:
\begin{itemize}
\item 
If the cell $(\alpha,\beta,\gamma)$
is not in $T{\cup} R$, then $A(\alpha,\beta,\gamma)=B(\alpha,\beta,\gamma)=H(\alpha,\beta,\gamma)$.
\item
If $(\alpha,\beta,\gamma)\in T$, then $A(\alpha,\beta,\gamma)=1$,
and $B(\alpha,\beta,\gamma)=0$.
\item
If $(\alpha,\beta,\gamma)\in R$, then 
$A(\alpha,\beta,\gamma)=0$, and $B(\alpha,\beta,\gamma)=1$.
\end{itemize}
We claim that the arrays $A$ and $B$ belong to $\Delta_n$. In
every line that meets $T\cup R$, both $A$ and $B$ have a unique $1$ entry.
All the other lines sum up to $1$ because $H\in \Delta_n$. 
Indeed, $A\ne B$ and $H=\frac{1}{2} \cdot A + \frac{1}{2}\cdot B$.
\end{proof}

We now turn to the construction. We first outline our work plan, and then proceed with the detailed step-by-step description.

\subsection{Work plan}
Our construction proceeds 
layer by layer from the bottom up and comprises five stages.
In the first stage we construct 
the lowest $k\approx \frac{n}{10}$ layers of the array. 
Subsequent layers are constructed with an eye to this initial stage.

As usual, a Hamiltonian cycle in a graph is a cycle that visits each vertex exactly once.
By a slight abuse of the language, we also refer to the adjacency matrix of a Hamiltonian cycle in $G$ or in $\overline{G}$ as a {\em Hamiltonian cycle}. 
Similarly, we refer to the adjacency matrix of a cycle in $G$ or in $\overline{G}$ as a cycle.

Additionally, a \hypertarget{standardcycle}{{\em standard cycle}} is a $t\times t$ matrix $Z$ where $Z(\alpha,\alpha)=Z(\beta,\beta+1)=Z(t,1)=\frac 12$ for all $\alpha=1,\ldots, t$ and $\beta=1,\ldots, t-1$, and zeros elsewhere.

We turn to the construction plan (see \autoref{fig:overview_construction}):
\begin{enumerate}
\item 
Here is what we do in slice $(1,k)$. Let $P$ be the
north-west $\big\lfloor \frac{n}{2} \big\rfloor\times \big\lfloor \frac{n}{2} \big\rfloor\times k$ sub-array of $H$. 
The south-east $\big\lceil \frac{n}{2} \big\rceil\times \big\lceil \frac{n}{2} \big\rceil\times k$ sub-array is called $Q$, see \autoref{fig:quarters}. In
each of the first $k$ layers, we place one
Hamiltonian cycle in $P$ and one in $Q$. 
These $2k$ cycles are chosen such that they are all disjoint, and each meets the main diagonal of $D(H)$, i.e., the shafts $H(t,t,\cdot)$.  
Consequently, the support of every shaft in the $(1,k)$-slice is either empty or a singleton.
In other words, every cell in $D^k(H)=D(H^{(1,k)})$ is either $0$ or $\frac{1}{2}$.

We describe next what we do in $P$, with $Q$ being treated essentially identically.
The submatrix of $D$ that corresponds to $H$'s north-west part $P$ is called $D(P)$. 
We associate with $P^{(1,i)}$ a bipartite graph $G^i = G^i_P = \langle X\cup Y, E^i\rangle$, 
where $|X|=|Y|=\big\lfloor \frac{n}{2} \big\rfloor$, 
and $E^i:=\{(x,y)\in X\times Y~|~D(x,y)=0\}$.

For every $i=1,\ldots,k$ we pick\footnote{We also refer to submatrices as minors.} 
a submatrix $C_i$ of size $c_i\times c_i$
along the main diagonal (\autoref{fig:quarterP}) 
and place a standard cycle there (\autoref{fig:standard cycle}). 
Each of the integers $c_i$ equals either $4$ or $5$ (more on this below)
and the $C_i$ minors induce a partition of $P$'s main diagonal. 
Additionally, we select a Hamiltonian cycle in the graph $G^{i}$
on the vertices not corresponding to the cells in $C_i$ (\autoref{fig:stage-1-ham-cycle}).
Subsequently, we merge the two parts into a single cycle via a switching step (\autoref{fig:stage-1-mix}).

A crucial ingredient in our proof is a result of Cuckler and Kahn \cite{cuckler2009hamiltonian}, which implies that this process yields a very large number of Hamilton cycles from which to choose. \label{work_plan_stage_1}
\item
Layer $(k+1)$ "glues" the disjoint cycles from lower layers. It
consists of a single cycle $L_{k+1}$ that includes 
all cells of the main diagonal of 
layer $H^{k+1}$. Consequently, the graph $U_{k+1}$ is connected. This is because
each cycle of lower layers has a $U_{k+1}$-edge to $L_{k+1}$ 
along the main diagonal of $D^{k+1}(H)$.

To guarantee that for $\ell > k+1$, the graph $U_{\ell}$ is connected, 
it suffices that
every connected component of $U_{\ell}$ at layer $\ell$
has at least one $U_\ell$-neighbor in a lower layer.
It follows that the resulting graph $U_H$ is connected, so we can
apply \cref{lemma:who_is_vertex} to the final construction.
\item
Layer $(k+2)$ is constructed such that the corresponding graph $U_{k+2}$ is non-bipartite, by introducing an odd cycle (\cref{prop:bi-colored connected graph}).
This layer comprises one permutation matrix 
supported on $D$'s $0$-cells and one on its $\frac 12$-cells.
Given the connectivity of the graph $U_H$ of the final construction, this stage ensures it is also non-bipartite, thus satisfying the conditions of \cref{lemma:who_is_vertex}, and 
implying that $H$ is a vertex.
\item
Each layer $i=k+3,\ldots, n-k$ comprises two perfect matchings, one in $G$ and one in $\overline{G}$.
Notice that the matrix $D^{n-k}$ has no $0$-cells.
This leads us to the final stage.\label{work_plan_stage_4}
\item
In each layer $i=n-k+1, \ldots n$ we use two perfect matchings
in $\overline{G}$. This stage continues until we complete the construction.
\label{work_plan_stage_5} 
\end{enumerate}

\begin{figure}[h!]
    \centering
    \includegraphics[width=0.65\linewidth]{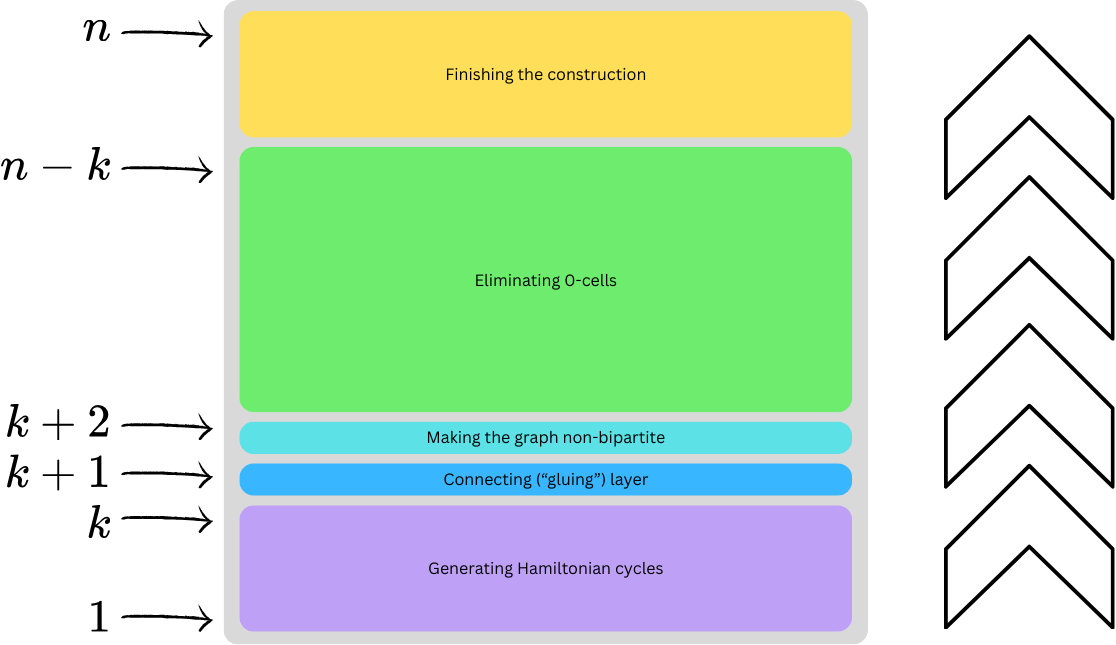}
    \caption{The Construction - It has $5$ stages and proceeds from the bottom up. Stage $1$ - generating Hamiltonian cycles, stage $2$ - connecting ("gluing") layer, stage $3$ - making $U_{k+2}$ non-bipartite, stage $4$ - eliminating $0$-cells of $D(H)$, stage $5$ - eliminating $\frac{1}{2}$-cells of $D(H)$ and completing the construction.}
    \label{fig:overview_construction}
\end{figure}

\subsection{Stage 1 - Generating Hamiltonian cycles: Layers $1, \dots ,k$}
We begin by splitting the $(1,k)$-slice of $H$ into four parts as follows (see \autoref{fig:quarters}):

\begin{itemize}
    \item $H(1\dots\big\lfloor\frac{n}{2}\big\rfloor, 1\dots\big\lfloor\frac{n}{2}\big\rfloor,1\dots k)$, called $P$.
    \item $H(\big\lfloor\frac{n}{2}\big\rfloor+1\dots n, \big\lfloor\frac{n}{2}\big\rfloor+1\dots n,1\dots k)$, called $Q$.
    \item $H(1\dots\big\lfloor\frac{n}{2}\big\rfloor, \big\lfloor\frac{n}{2}\big\rfloor+1\dots n,1\dots k)$, called $NE$.
    \item $H(\big\lfloor\frac{n}{2}\big\rfloor+1\dots n, 1\dots\big\lfloor\frac{n}{2}\big\rfloor,1\dots k)$, called $SW$.
\end{itemize}
Next, we describe the construction of $P$. We similarly deal with $Q$.
Work on the other two parts of the array $NE$ and $SW$ commences only later on.

\begin{figure}[h!]
    \hspace*{2.95cm}
    \includegraphics[width=0.52\linewidth]{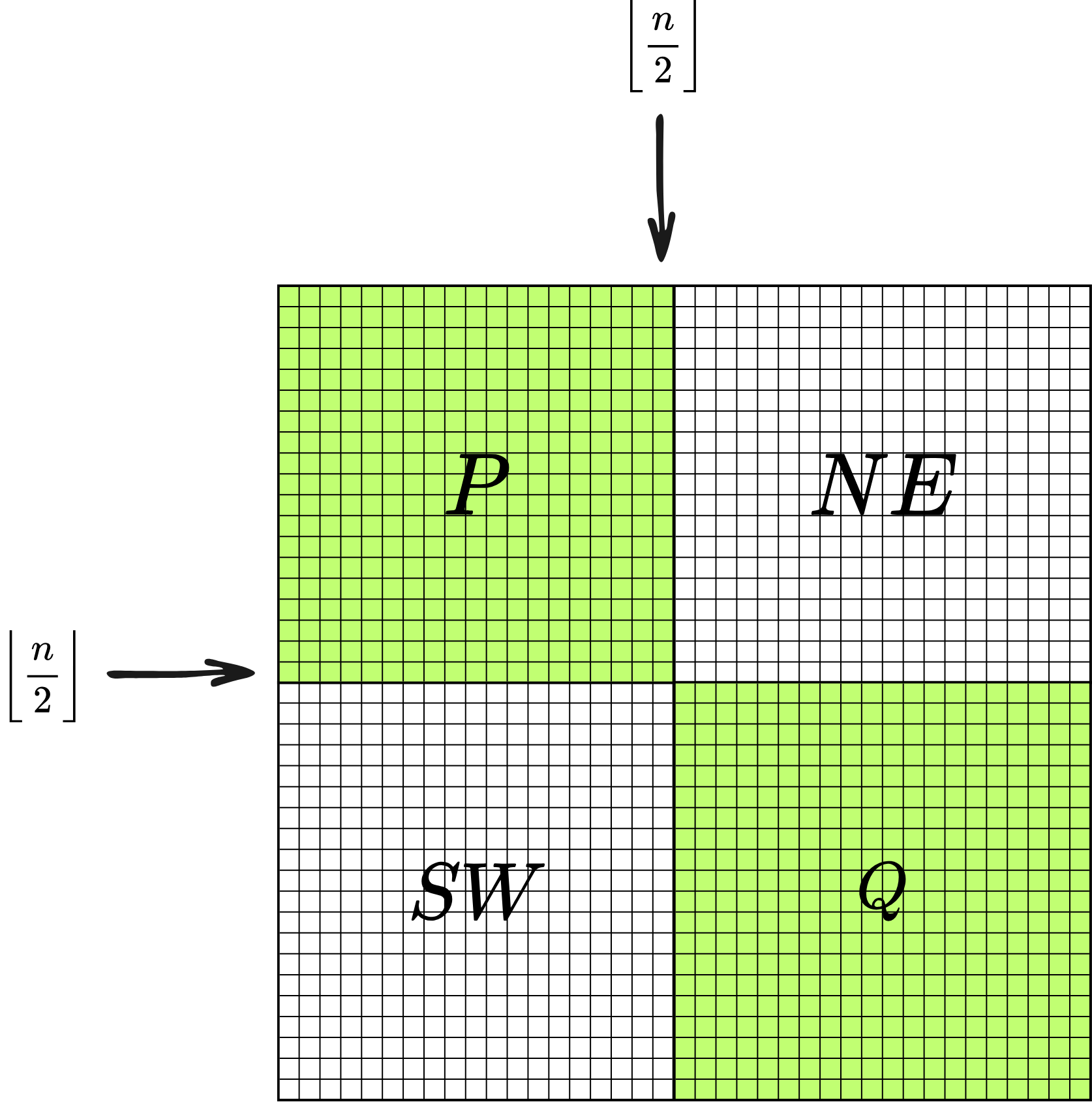}
    \caption{Split $H^{(1,k)}$ into: $P$ - the  north-west $\big\lfloor\frac{n}{2}\big\rfloor\times \big\lfloor\frac{n}{2}\big\rfloor\times k$ subarray,
    $Q$ - the south-east $\big\lceil\frac{n}{2}\big\rceil\times \big\lceil\frac{n}{2}\big\rceil\times k$ subarray, 
    $NE$ - the north-east $\big\lfloor\frac{n}{2}\big\rfloor\times \big\lceil\frac{n}{2}\big\rceil\times k$ subarray,
    $SW$ - the south-west  $\big\lceil\frac{n}{2}\big\rceil\times \big\lfloor\frac{n}{2}\big\rfloor\times k$ subarray.}
    \label{fig:quarters}
\end{figure}

\subsubsection{Construction of $P$}
As mentioned, we construct
the sub-array $P$ layer-by-layer. We express its side length
as $\big\lfloor \frac{n}{2} \big\rfloor=\sum_{i=1}^k c_i$, where 
$4\le c_1 \le c_2 \le \ldots \le c_k \le 5$
(\cref{prop:frob}).
At each layer $1\le i\le k$ we set aside a sub-matrix (aka minor) 
$C_i$ of size $c_i\times c_i$, along the main diagonal.
These minors are disjoint and together cover the main diagonal
of $D=D(P)$.
Namely, for every $t=1,\ldots, \big\lfloor\frac{n}{2}\big\rfloor$,
the $t$-th cell of the diagonal, $D(t,t)$, is in exactly one of these minors.
This is possible, since $\big\lfloor\frac{n}{2}\big\rfloor=\sum_{i=1}^k c_i$,
see \autoref{fig:quarterP}.

At each layer $i=1,\dots,k$,
we place a standard cycle in the chosen $C_i$ minor,
as well as a Hamiltonian cycle, $L_i$, in the rest of the layer's cells.
Then we combine these two parts into a single Hamiltonian cycle $L'_i$ via a switching step.
The selection of $L_i$ for $i\leq k$ relies 
on work by Cuckler and Kahn \cite{cuckler2009hamiltonian}. We first note:

\begin{prop}\label[proposition]{prop:frob}
Every integer $N\ge 12$ is expressible as $N=5a+4b$,
where $a,b\ge 0$ are integers, and $b\le 4$.
\end{prop}
\begin{proof}
This is easily verified
for $N=12,13,14,15, 16$. For larger $N$, argue by induction in steps of $5$.
\end{proof}

We deal with $Q$ likewise, with minor sizes 
$4\le c_1' \le c_2' \le \ldots \le c_k' \le 5$.
If $n$ is even, then $c_i'=c_i$
for all $i$. If $n$ is odd, $Q$'s side length exceeds that of $P$
by $1$, and we let $c_1':=c_1+1$ and $c_i'=c_i$ for all $i\ge 2$. 
A slight modification is needed when $5|\big\lfloor \frac{n}{2} \big\rfloor$ and $n\geq 40$. In this 
case, we let $c_1=\ldots=c_5=4,\ c_6=\ldots=c_k=5$ and 
$c_1'=\ldots=c_4'=4,\ c_5'=\ldots=c_k'=5$. 
\begin{figure}[h!]
    \hspace*{2.8cm}
    \includegraphics[width=0.53\textwidth]{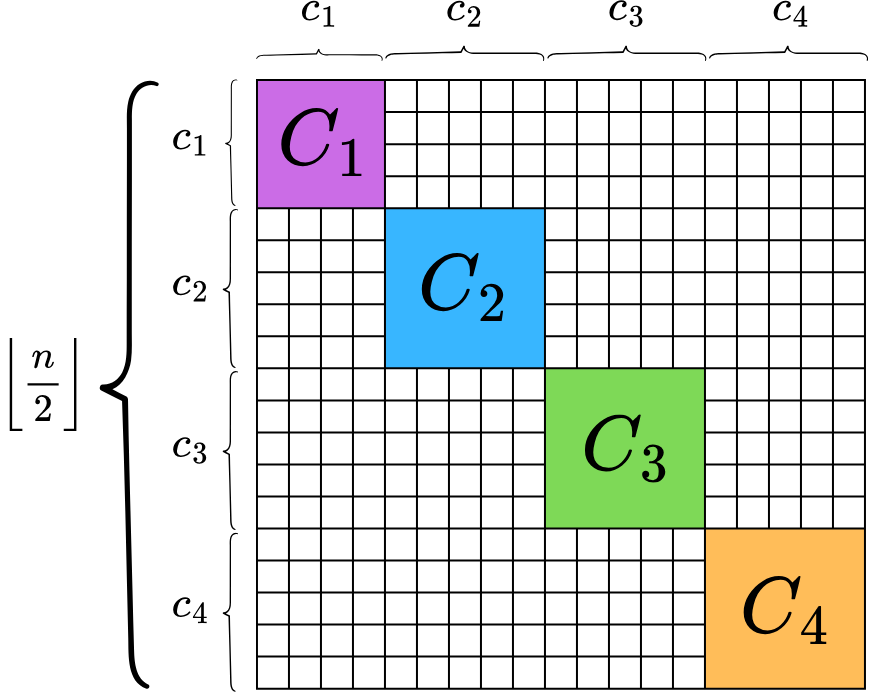}
    \caption{Reserved places for the diagonal minors $C_i$ of size $c_i\times c_i$ for layers $i=1,\dots,k$ in $D(P)$.}
    \label{fig:quarterP}
\end{figure}

Next, we describe the construction of each layer in $P$.

\paragraph{Construction of layer $i$ in $P$}

The construction of layer $i$ in the sub-array $P$ is done in three steps:

\begin{itemize}
    \item[1.] Place a \hyperlink{standardcycle}{standard cycle} in minor $C_i$ (see \autoref{fig:standard cycle}).
    
    \begin{figure}[ht]
    \hspace*{4.2cm}
    \includegraphics[width=0.44\textwidth]{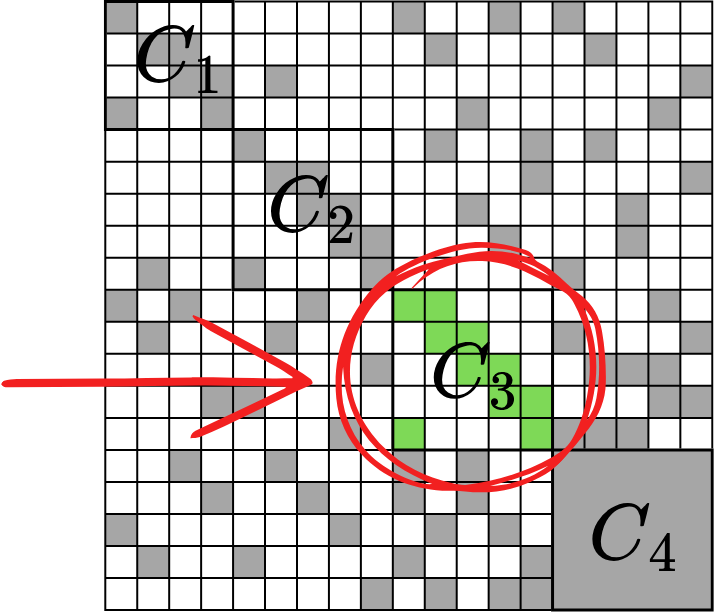}
    \caption{To start the construction of layer $3$ (step $1$), we place a standard cycle within minor $C_3$, highlighted in green. The Hamiltonian cycles from the previous layers, as well as the minors preserved for the subsequent layers, are colored in gray.}
    \label{fig:standard cycle}
    \end{figure}
    \item[2.] \label[2]{section-2-stage-1}
    Now, let $S^i$ be the $\left(\big\lfloor \frac{n}{2} \big\rfloor-c_i\right) \times \left(\big\lfloor \frac{n}{2} \big\rfloor-c_i\right) \times k$ sub-array of $P$, defined by removing the $x$-walls and $y$-walls that intersect with $C_i$ (see \autoref{fig:def S}).
    During this stage, we refer to $D^i$ as $D(S^i)$.  
    \begin{figure}[h!]
    \centering
    \includegraphics[trim=0cm 0cm 0cm 2cm, clip,width=0.52\textwidth]{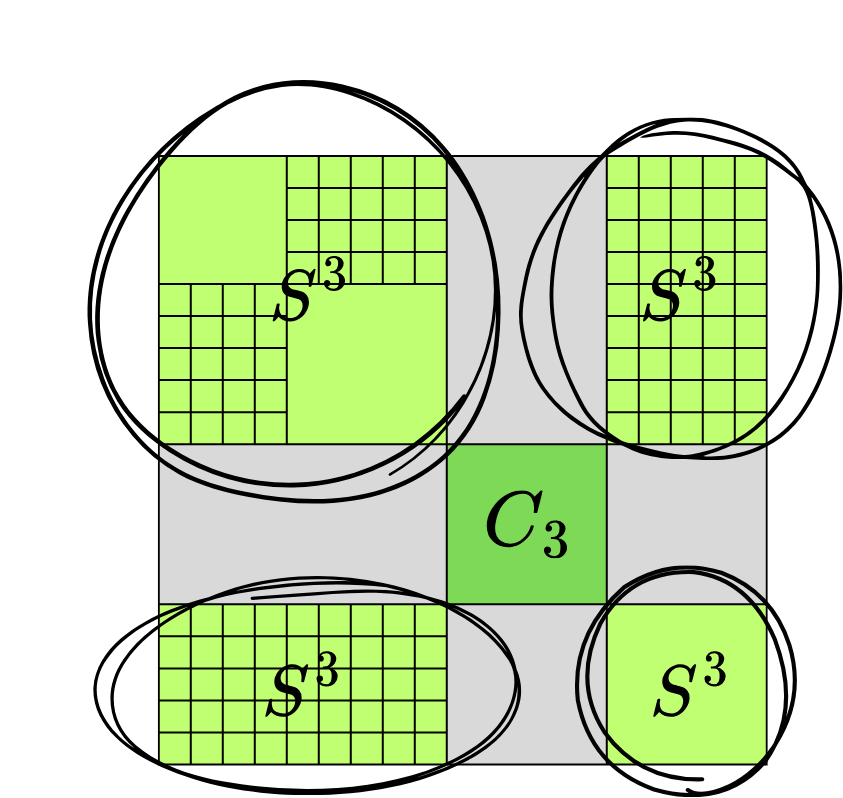}
    \caption{The subarray $S^3\subset P$ (light green) defined by removing the $x$-walls and $y$-walls of $P$ that intersect with the minor $C_3$ (colored in gray and dark green), yielding the size $\left(\big\lfloor \frac{n}{2} \big\rfloor-c_3\right) \times \left(\big\lfloor \frac{n}{2} \big\rfloor-c_3\right) \times k$.}
    \label{fig:def S}
    \end{figure}

Associate with $S^i$ a bipartite graph $G^i=\langle X\cup Y, E^i\rangle$, where the vertex sets $X$ and $Y$ represent the rows and columns of $D^i=D(S^i)$, respectively, thus $|X|=|Y| = \big\lfloor \frac{n}{2} \big\rfloor-c_i$. An edge $(x,y)\in X\times Y$ is included in $E^i$ iff $D^i(x,y)=0$.

We find a Hamiltonian cycle in $G^i$ using a condition due to Moon and Moser \cite{moon1963hamiltonian}.
Theirs is a bipartite counterpart of the better-known Dirac's Theorem \cite{dirac1952some}.
It states that every bipartite graph $\langle X\cup Y, E\rangle$  
with $|X|=|Y|=m$ where all vertex degrees are at least $m/2$ has a Hamiltonian cycle.
    
Here, the degree of vertex $x\in X$ in $G^i$, is the number of $0$-cells in row $x$ in $D^i$.
The total number of cells in this row is $\big\lfloor \frac{n}{2} \big\rfloor-c_i$. 
Each Hamiltonian cycle in a 
lower layer turns at most two of them into $\frac{1}{2}$-cells. 
Also, this row meets some 
$C_j$-minor along the diagonal, discounting $c_j$ cells.
All told, this yields 
$\text{deg}(x)\geq\big\lfloor \frac{n}{2} \big\rfloor-c_i-c_j - 2(i-1)\geq \big\lfloor \frac{n}{2} \big\rfloor-c_i - 2i-3$ (see \autoref{fig:a row in P}). Let us verify that Moon and Moser's condition holds.

Namely, we need to show that 
$$\text{deg}(x)\geq \frac{|X|}{2} =\frac{\big\lfloor \frac{n}{2} \big\rfloor-c_i}{2},$$ 
which simplifies to
$$i\leq \frac{\big\lfloor \frac{n}{2} \big\rfloor-c_i -6}{4}.$$ Since 
$i\leq k \approx \frac{n}{10}$, this 
is the case for $n\ge 114$.
    
Let $L_i$ be the adjacency matrix of the Hamiltonian cycle in $G^i$ and 
define the $i$-th layer as follows (see \autoref{fig:stage-1-ham-cycle}):
$$P(a,b,i)\xleftarrow{}\frac{1}{2}\text{~for every~} (a,b)\in L_i
{\text{~and zero for every~} S^i\setminus L_i}.$$ 

\begin{figure}[h!]
    \centering
    \includegraphics[width=0.37\textwidth]{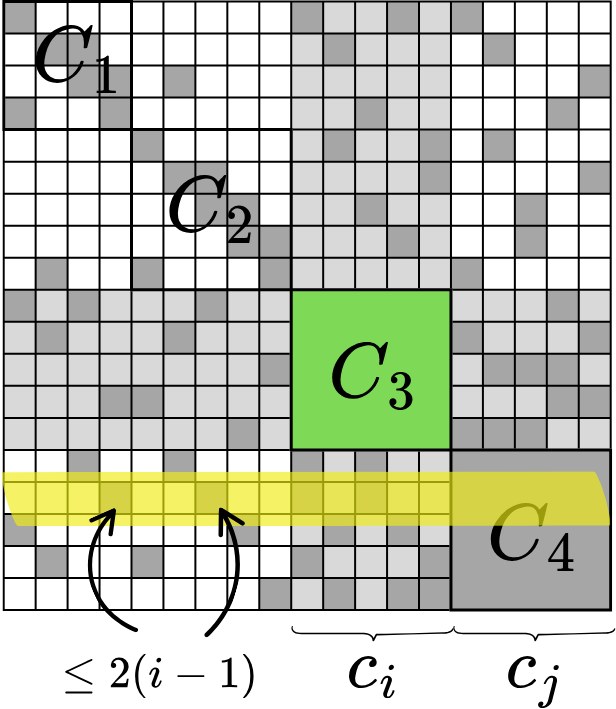}
    \caption{A lower bound on the vertex degrees in $G^i$ follows, provided that we find many $0$ entries in each row of $D^i$.
Nonzero entries are marked in dark gray.
A row in the sub-array $S^i$ has $\big\lfloor \frac{n}{2} \big\rfloor-c_i$ entries.
At most $2(i-1)$ non-zero entries in a row
are due to choices made at lower layers.
Also, each row meets some diagonal minor $C_j$ ($j\ne i$), and the corresponding $c_j$ cells do not count toward the degree. 
For example, in the row highlighted in yellow, $i=3$ and $j=4$, yielding a lower bound of
$\big\lfloor \frac{n}{2} \big\rfloor-c_3-c_4-2\cdot 2$ on the degree.}
    \label{fig:a row in P}
\end{figure}

\begin{figure}[h!]
    \centering
    \includegraphics[width=0.42\textwidth]{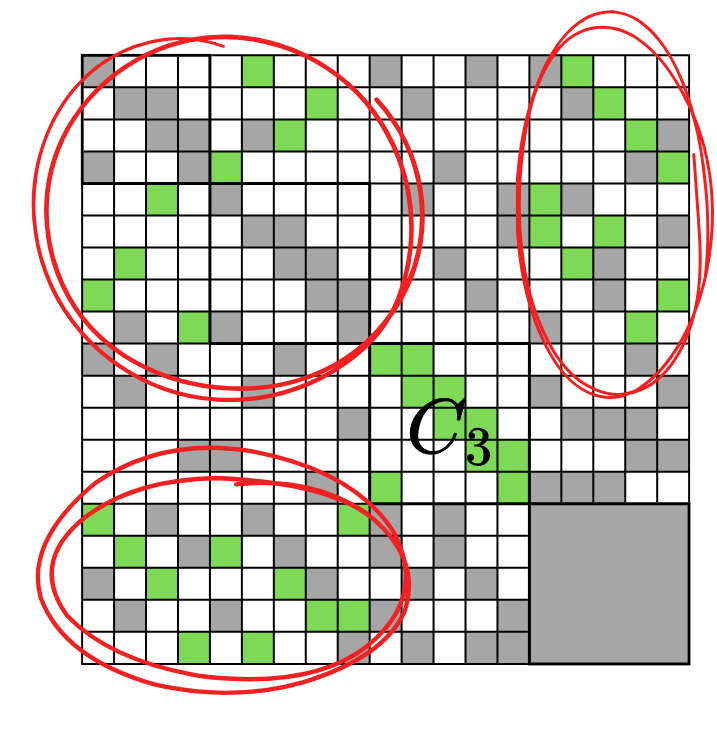}
    \caption{
Step $2$ applied to layer $3$. Circled in red - the sub-array $S^3$ where the Hamiltonian cycle $L_3$ is placed. Its cells are in green. Minor $C_3$ hosts the previously established standard cycle. In gray - cells of the support of previous layers and minors that are reserved for future steps.}
    \label{fig:stage-1-ham-cycle}
\end{figure}
\item[3.] 
Presently, the $i$-th layer of $P$ contains two disjoint cycles, the standard cycle in $C_i$ and the cycle $L_i$. 
We wish to combine them into a single Hamiltonian cycle $L_i'$, while respecting the requirement that we are constructing an array in $\cal H$. 
To this end we seek two $\frac 12$-cells $(x_1,y_1), (x_2,y_2)\in D^i(P)$ such that: 
\begin{equation}\label{eq:merge}
\begin{split}
(x_1,&y_1) \in \left(C_i \setminus \text{the main diagonal}\right) \text{~and~} (x_2,y_2) \in L_i .\\
&\text{Also, both~} (x_1,y_2) \text{~and~} (x_2,y_1) \text{~are~} 0-\text{cells}.
\end{split}
\end{equation}
We then perform the following switch: 
$$P(x_1,y_1, i),\ P(x_2,y_2, i) \xleftarrow{} 0
\text{~and~} 
P(x_1,y_2,i) ,\ P(x_2,y_1, i) \xleftarrow{} \frac{1}{2}.$$
This merges the two cycles into a single Hamiltonian cycle $L_i'$, without changing the total row and column sums, 
while all entries of $D^i(P)$ are $0$ or $1/2$, as planned.

The following counting argument
shows that such cells exist. We first
fix a $\frac 12$-cell $(x_1,y_1)$ in $(C_i \setminus \text{the main diagonal})$,
and search for corresponding $(x_2,y_2) \in L_i$ such that conditions (\ref{eq:merge}) hold.
Let $\overline{X_2}$ (resp.\ $\overline{Y_2}$)
be the set of non-zero cells in $D^i$, excluding those in $C_i$, that share a column (resp.\ row) with the fixed cell $(x_1,y_1)$ and thus cannot serve as $(x_2,y_1)$ (resp.\ $(x_1,y_2)$):
$$\overline{X_2}:=\{ (x,y_1)\not\in C_i \mid D^i(x,y_1) > 0 \}\ ;\ \overline{Y_2}:=\{ (x_1,y)\not\in C_i \mid D^i(x_1,y)> 0\}$$

For each cell in $\overline{X_2} \cup \overline{Y_2}$ there are two cells in the Hamiltonian 
cycle $L_i$ that violate condition (\ref{eq:merge}), because
each line in the sub-array $D(S^i)$ contains exactly two entries of the cycle.
In this way we eliminate at most
$2\cdot \left(\left|\overline{X_2}\right|+\left|\overline{Y_2}\right|\right)$ cells of $L_i$ (see \autoref{fig:stage-1-mix-new}).
It remains to show that 
$$|L_i| > 2\cdot \left(\left|\overline{X_2}\right|+\left|\overline{Y_2}\right|\right).$$

Indeed, 
$$\left|\overline{X_2}\right| = \left|\overline{Y_2}\right| = 2(i-1),$$ 
because
for each $\ell<i$ the cycle $L'_\ell$ contributes two non-zero cells (outside of $C_i$)
to $\overline{X_2}$, resp.\ $\overline{Y_2}$.
The support size of $L_i$ is
$$\left|L_i\right| = 2\cdot \left(\Big\lfloor \frac{n}{2} \Big\rfloor-c_i\right),$$
being a Hamiltonian cycle in the sub-array $D(S^i)$ of side length 
$\left(\big\lfloor \frac{n}{2} \big\rfloor-c_i\right)$.
Finally, for $i\leq k \approx \frac{n}{10}$ and $c_i\leq 5$. So, if $n>15$, then:
$$\left|L_i\right| - 2\cdot \left(\left|\overline{X_2}\right|+\left|\overline{Y_2}\right|\right) = 
2\cdot \Big\lfloor \frac{n}{2} \Big\rfloor -2\cdot c_i -8 (i-1)  \ge
\frac{n}{5} - 3 > 0 $$ 
Thus, there exists $(x,y)\in L_i$ s.t.\ $(x,y_1)\not\in \overline{X_2}$ and $(x_1,y)\not\in \overline{Y_2}$, enabling the desired merging operation, which yields a single Hamiltonian cycle $L_i'$, 
as illustrated in \autoref{fig:stage-1-mix}. 
\end{itemize}

\begin{figure}[ht]
    \centering
    \includegraphics[width=0.53\textwidth]{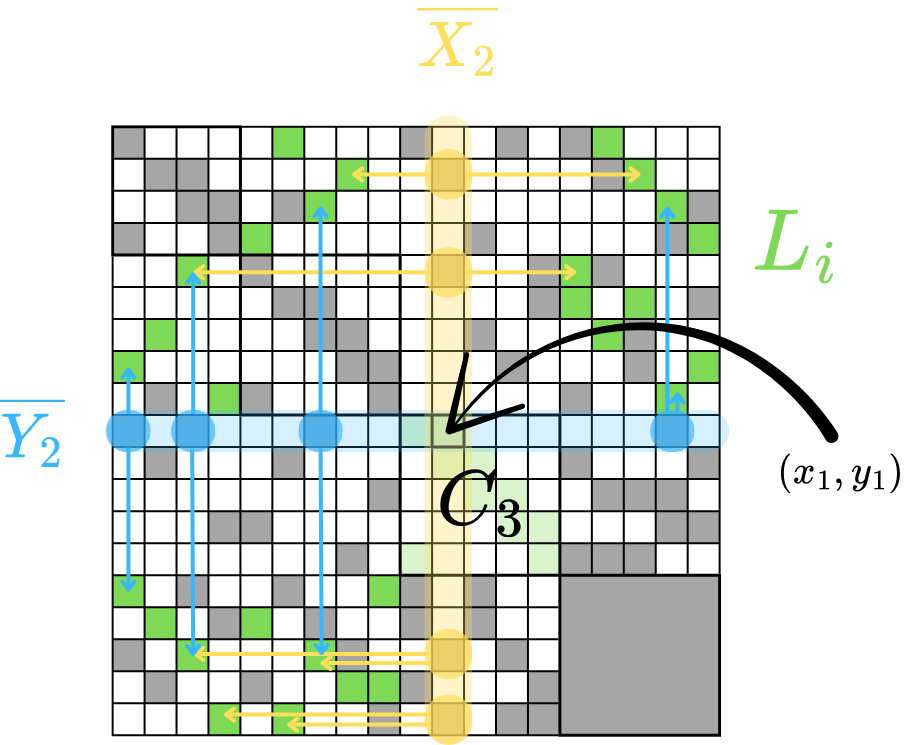}
    \caption{
    $(x_1,y_1)$ is a fixed $\frac{1}{2}$-cell in $C_i$ outside the main diagonal. 
    $\overline{X_2}$ (Yellow) is the set of non-zero cells in column $y_1$ (excluding $C_i$).
    Likewise, $\overline{Y_2}$ (Blue) is the set of non-zero cells in row $x_1$ (excluding $C_i$). 
    Arrows indicate how entries in $\overline{X_2}\cup \overline{Y_2}$ eliminate potential swap candidates in the Hamiltonian cycle $L_i$ (in green) at their same columns or rows.
    The proof shows that green cells outnumber the eliminated candidates, so that $L_i$ must contain a cell that is eligible for a swap.}
    \label{fig:stage-1-mix-new}
\end{figure}

\begin{figure}[h!]
    \hspace*{3.3cm}
    \includegraphics[width=0.45\textwidth]{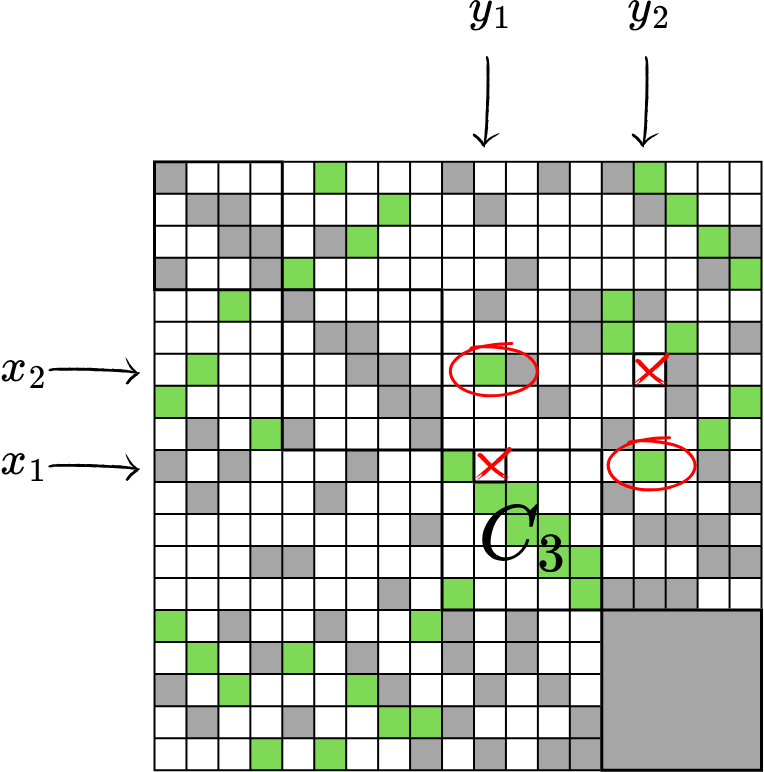}
    \caption{
    With two $\frac 12$-cells $(x_1,y_1) \in \left(C_i \setminus \text{the main diagonal}\right)$ and $(x_2,y_2) \in L_i$, s.t.\ both $(x_1,y_2)$ and $(x_2,y_1)$ are $0$-cells, we apply the switching operation indicated in red (step $3$).}
    \label{fig:stage-1-mix}
\end{figure}

\subsection{Stage 2 - Connecting the Construction: Layer $k+1$}
We have constructed so far $L'_1, L'_2, \ldots, L'_k$, disjoint Hamiltonian cycles
in $P$, and, similarly, $L''_1, L''_2, \ldots, L''_k$, disjoint Hamiltonian cycles in $Q$, two cycles per layer. For each $i\leq k$, $L'_i$ and $L''_i$ meet
several entries of main diagonal of $D^k=D^k(H)$. 
The purpose of the cycle $L_{k+1}$ we choose at layer $k+1$ is to "glue" all previously constructed cycles into a single connected structure. To begin, we construct:
\begin{itemize}
    \item The main diagonal of $H^{(k+1)}$ 
    $$H(i,i,k+1) \xleftarrow{} \frac{1}{2}\text{~for~}i=1,\ldots,n$$
\end{itemize}

Next, we complete the main diagonal of $H^{(k+1)}$ into a single cycle using the previously unused parts of $H$. 
Namely, $NE^{(k+1)}=H(1\dots\big\lfloor \frac{n}{2} \big\rfloor, \big\lfloor \frac{n}{2} \big\rfloor+1\dots n,k+1)$ and $SW^{(k+1)}=H(\big\lfloor \frac{n}{2} \big\rfloor+1\dots n, 1\dots\big\lfloor \frac{n}{2} \big\rfloor,k+1)$, see \autoref{fig:stage-2-even-case}.

The specifics depend on the parity of $n$. For even $n$, we define:
\begin{itemize}
    \item In $NE^{(k+1)}$:
    $$H(i,i+\frac{n}{2},k+1) \xleftarrow{} \frac{1}{2}\text{~for~}i=1,\dots,\frac{n}{2}$$
    \item In $SW^{(k+1)}$
    $$H(\frac{n}{2}+i,i+1,k+1) \xleftarrow{} \frac{1}{2} \text{~for~}i=1,\dots,\frac{n}{2}-1$$
    $$H(n,1,k+1) \xleftarrow{} \frac{1}{2}$$
    \item[] and zeros elsewhere. 
\end{itemize}  

\begin{figure}[h!]
    \centering
    \includegraphics[width=0.5\linewidth]{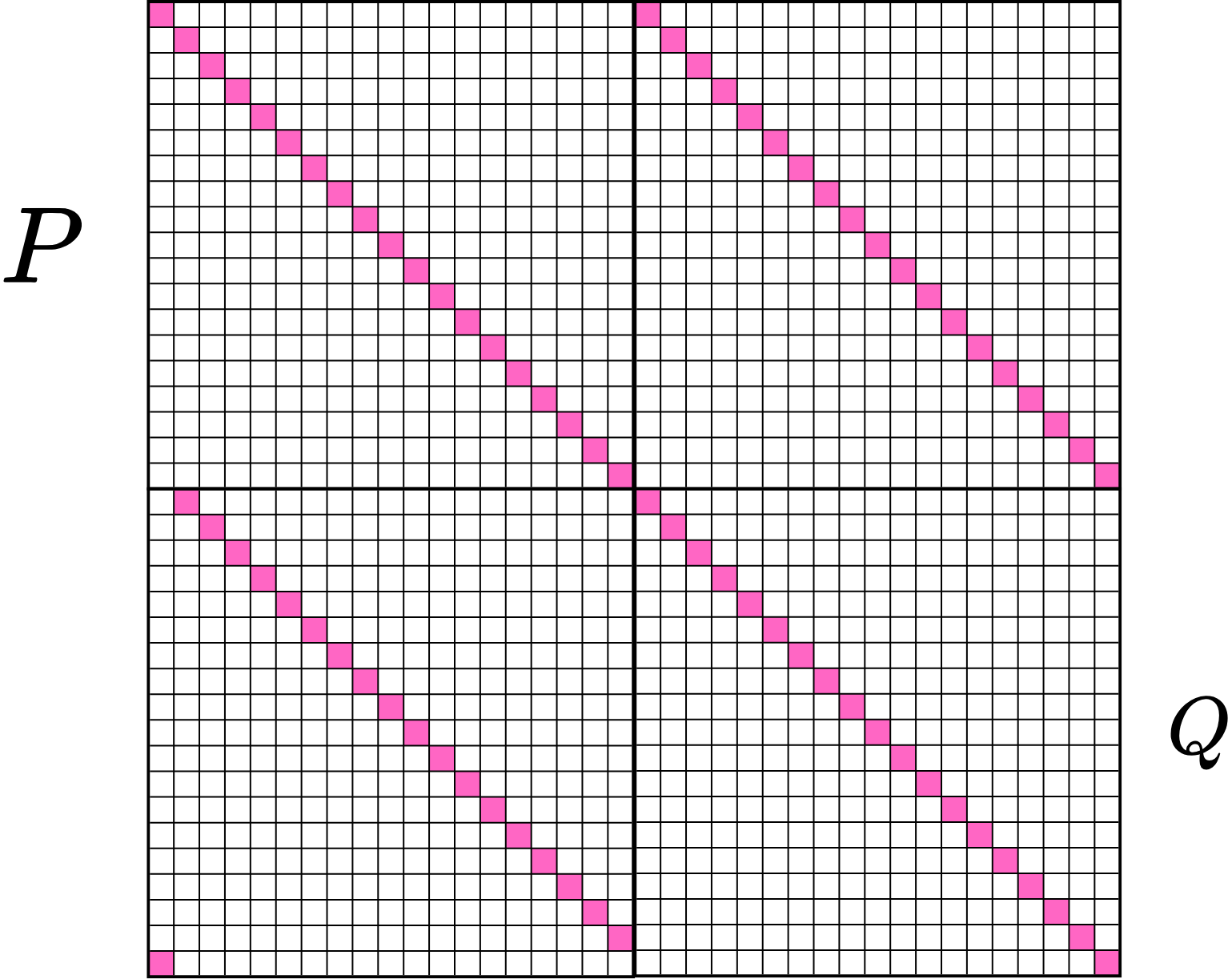}
    \caption{The construction of layer $k+1$ when $n$ is even.}
    \label{fig:stage-2-even-case}
\end{figure}

If $n$ is odd, the number of rows and columns in $NE^{(k+1)}$ and $SW^{(k+1)}$ differ by one.
Therefore, to complete the cycle we must use a cell from $Q^{(k+1)}$.
This is possible, as each row and column of $Q^{(k+1)}$ has available $0$-cells.
We omit the technical details and provide instead a graphical illustration of the necessary steps, see \autoref{fig:stage-2-odd-case}. 
\begin{figure}[h!]
    \centering
    \includegraphics[width=0.55\linewidth]{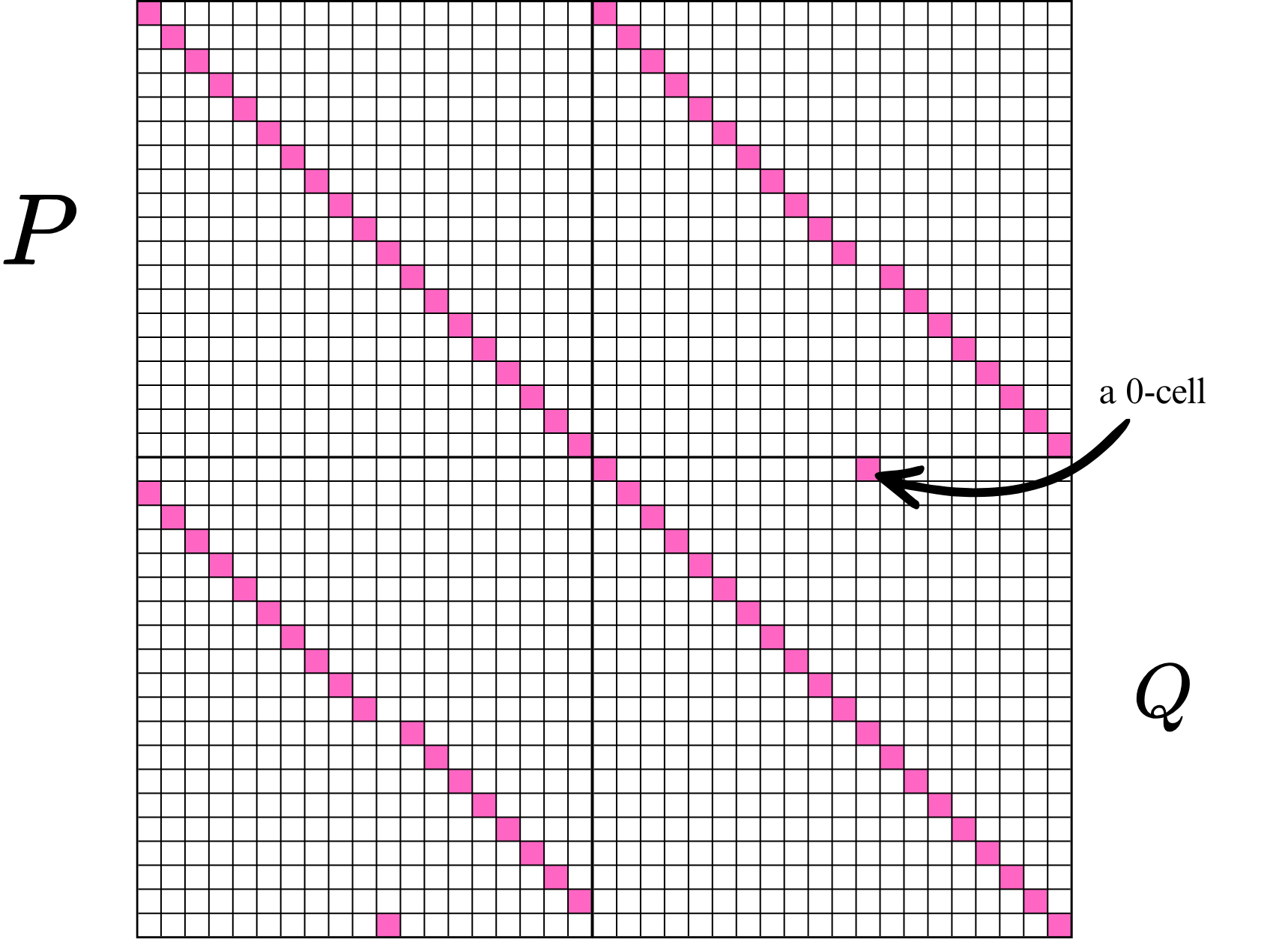}
    \caption{The construction of layer $k+1$ when $n$ is odd. In this case we use a $0$-cell in $Q$.}
    \label{fig:stage-2-odd-case}
\end{figure}

Note that in this stage, we turn the $\frac{1}{2}$-cells of the main diagonal of $D$ into $1$-cells and another $n$ cells turn from $0$-cells into $\frac{1}{2}$-cells. 

\subsection{Stage 3 - Making $U_{k+2}$ Non-Bipartite: Layer $k+2$}
In this stage we define layer $k+2$ such that the associated graph $U_{k+2}$ 
is non-bipartite. As a consequence, the graph associated with the final construction 
is also non bipartite. Provided that this final graph is connected, by 
\cref{lemma:who_is_vertex}, we obtain the desired vertex of the polytope $\Delta_n$.

Recall the matrix $D^\ell=D^\ell(H)$ that is associated with slice $(1,\ell)$ of $H$.
Namely, $D (i,j)= \sum_{\nu=1}^\ell H(i,j,\nu)$.

The graph $U_{k+1}$ consists of $2k+1$ cycles of even length that are 
disjoined except at the vertices that correspond to the main diagonal of $D^{k+1}$.
Had it contained an odd cycle (which it does not...) we were done.
Otherwise, $U_{k+1}$ is a connected bipartite graph,
whence, by \cref{prop:bi-colored connected graph}, it admits a unique 2-vertex-coloring 
(up to swapping the two colors), denoted by
red and blue, see \autoref{fig:stage-3-before}.

This vertex-coloring of $U_{k+1}$ induces a
$2$-coloring of the $\frac 12$-cells of $D^{k+1}$. Let $\varepsilon$ be such a cell.
In the shaft of $H$ corresponding to $\varepsilon$, there is a single entry of $\frac 12$, which we view as a vertex $v$ of $U_{k+1}$.
We color the cell $\varepsilon$ in $D^{k+1}$ by the color of $v$ in $U_{k+1}$, red or blue. 

Observe that for every $\ell\leq k+1$, each row (resp.\ column) in layer $H^\ell$ contains exactly two $H$-entries of $\frac 12$. 
They correspond to two neighboring vertices in $U_{k+1}$, and therefore 
one of them is red, and the other is blue. 
The only shafts with two non-zero cells reside along the main diagonal shafts of the slice $H^{(1,k+1)}$, their corresponding vertices in $U_{k+1}$ are of distinct colors.
Consequently, each row (resp.\ column) of $D^{k+1}$ has a single $1$-cell,
the number of its $\frac{1}{2}$-cells is $2k$ and the remaining $n-2k-1$ cells equal $0$. 
Moreover, the $\frac{1}{2}$-cells in every row and column are evenly colored: $k$ red and $k$ blue.
 
In defining layer $k+2$ we rely on
the present colors of the $\frac{1}{2}$-cells in $D^{k+1}$.
Specifically, during the construction of the current layer, certain $\frac{1}{2}$-cells will turn into $1$-cells.
For each such cell in $D^{k+2}$, the corresponding shaft in $H^{k+2}$ will contain two non-zero entries, whose associated vertices in $U_{k+2}$ are oppositely colored. To avoid confusion and maintain consistency, throughout
this stage, we refer to those cells by their initial color, i.e.\ by the color of the vertex corresponding to the lower non-zero entry.

\begin{figure}[h!]
    \centering
    \includegraphics[width=0.45\linewidth]{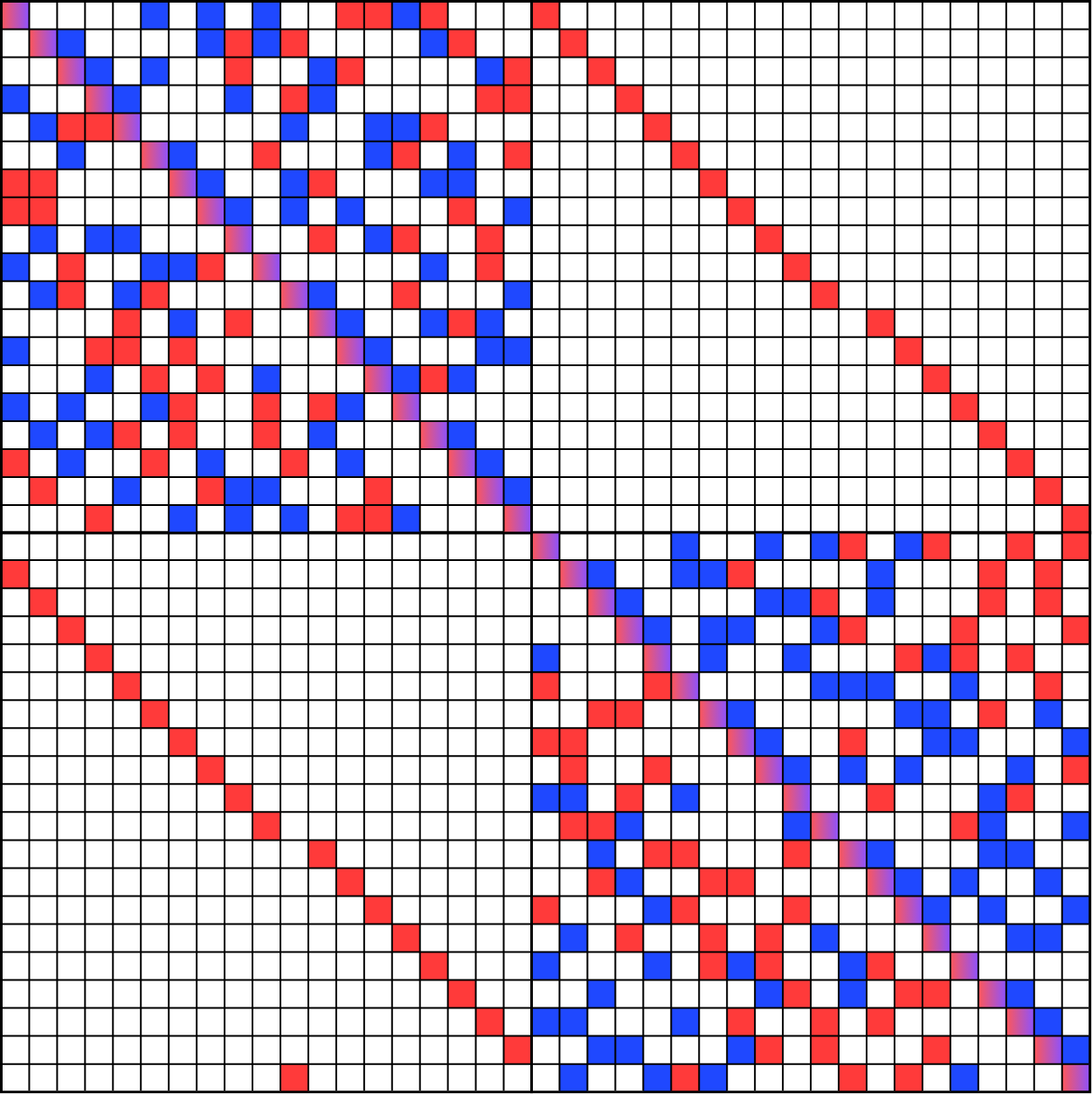}
    \caption{The 2-coloring of $D^{k+1}(H)$, the main diagonal consists of $1$-cells and the corresponding vertices in $U_{k+1}$ are bi-colored.}
    \label{fig:stage-3-before}
\end{figure}

To construct layer $k+2$, our plan is to place two permutation matrices $\overline{P}$ and $P$ as follows:
\begin{enumerate}
\item
We will choose
$\overline{P}$ as a permutation matrix that is supported
on the $\frac{1}{2}$-cells of $D^{k+1}$. We insist that
the entries of this permutation are of both colors.
\Cref{lemma:bi-chromatic} below shows that this condition holds for
asymptotically almost every choice of $\overline{P}$.
\item\label{cond:non_bipartite}
Likewise, we choose $P$ to be a permutation matrix that is supported
on the $0$-cells of $D^{k+1}$. Here we require that at least one of $P$'s
cells has two $\overline{P}$-{\em neighbors} of different colors.
(A $\overline{P}$-{\em neighbor} of a $0$-cell $\pi$ is a cell in  
$\overline{P}$ that resides in $\pi$'s row or column).
\Cref{lemma:distinct_neighbors} shows that this requirement is feasible.
\item 
With the introduction of the cells of $\overline{P}$ and $P$
to layer $k+2$ (see \autoref{fig:stage 3 - permutations}), the graph $U_{k+2}$ is non-bipartite, as the following
argument shows: We first consider the $\overline{P}$ cells and
color each one opposite to the color of the non-zero cell in the same shaft.
This is the unique extension of the $2$-coloring of $U_{k+1}$.
However, the condition stated in \Cref{cond:non_bipartite}
shows that the resulting $U_{k+2}$ is non-bipartite, due to the
cell that has both a blue and a red $\overline{P}$-neighbor.
\end{enumerate}

We first construct $\overline{P}$.
Recall $\overline{G}$, the bipartite graph with vertex sets $X$ and $Y$, 
the row and column sets of $D^{k+1}$, respectively, 
where $(x,y)$ is an edge iff $D^{k+1}(x,y)=\frac{1}{2}$. 
Let $M$ be the (bipartite) adjacency matrix of $\overline{G}$. 
Similarly, let $\overline{G_R}$ (resp.\ $\overline{G_B}$) be the subgraph of $\overline{G}$ that corresponds to the red (resp.\ blue) $\frac{1}{2}$-cells in $D^{k+1}$,
with adjacency matrices $M_R$ and $M_B$, respectively.

\begin{lemma}\label[lemma]{lemma:bi-chromatic}
Asymptotically almost every permutation matrix supported on the $\frac{1}{2}$-cells in $D^{k+1}$ is bi-colored.
\end{lemma}
\begin{proof}
Let $F$ be the event that a uniformly chosen permutation of the $\frac{1}{2}$-cells in 
$D^{k+1}$ is monochromatic. 
 
The graph $\overline{G}$ has exactly $\text{per}(M)$ perfect matchings.
Likewise, it has $\text{per}(M_R)$ resp.\ $\text{per}(M_B)$ all-red resp.\ all-blue perfect matchings.
The row and column sums of $M$ are $2k$, while in $M_R$ and $M_B$ they equal $k$. 
We apply claim \ref{claim:bound_d-regular}, (see below), to $M$, $M_R$ and $M_B$, and conclude:
$$\Pr(F)= \frac{\text{per }(M_R)+\text{per }(M_B)}{\text{per }(M)} < 2\cdot \frac{\Big((ek)^{1/k}\cdot \frac{k}{e}\Big)^n}{\Big( \frac{2k}{e} \Big)^n}\leq 2\cdot \bigg( \frac{(ek)^{1/k}}{2} \bigg)^n \xrightarrow{n\xrightarrow{}\infty}0.$$
The last limit is valid for $k \ge 4$, which does not affect the validity of
\cref{thm:main}.
\end{proof}

Next, we turn to the numerical claim used in the proof:

\newtheorem{claim}{Claim}[section]
\begin{claim}\label[claim]{claim:bound_d-regular}
Let $M$ be the $X\times Y$ adjacency matrix of a $d$-regular bipartite graph
$\langle X,Y; E\rangle$ with $|X|=|Y|=n$. Then:
$$\Big(\frac{d}{e}\Big)^n < \text{~per~}(M) < \Big(\frac{d}{e}\Big)^n\cdot(ed)^{n/d}$$
\end{claim}

\begin{proof}
We start with the upper bound. By the Brégman–Minc inequality \cite{bregman1973some}\cite{minc1963upper}
$$\text{per}(M)\leq \prod_{i=1}^n (d!)^{1/d} < \Big(\frac{d}{e}\Big)^n\cdot(ed)^{n/d},$$
where we use the fact that every row-sum of $M$ equals $d$, as well as
the following version of Stirling's formula (that holds for every $m>1$,
e.g., \cite{hummel1940note}) 
\begin{equation}\label{eq:m_factorial}
        \Big( \frac{m}{e}\Big)^m < m! < em\cdot \Big( \frac{m}{e}\Big)^m
\end{equation}

For the lower bound, we use Egorichev and Falikman's proof
of the van der Waerden's conjecture \cite{egorychev1981solution}\cite{falikman1981proof}. 
Since $\frac{1}{d}\cdot M$ is a doubly stochastic matrix it follows that:
$$\text{per}(M)= d^n\cdot \text{per}\Big(\frac{1}{d}\cdot M\Big) \geq d^n\cdot \frac{n!}{n^n} > \Big(\frac{d}{e}\Big)^n$$
where again we use \cref{eq:m_factorial}.
\end{proof}

So, let us pick one bi-chromatic perfect matching $\overline{P}$ of the $\frac{1}{2}$-cells in $D^{k+1}$ (see left panel of \autoref{fig:stage 3 - P}). 
We turn to construct the permutation matrix $P$ that is supported on the $0$-cells of $D^{k+1}$.

\begin{lemma}\label[lemma]{lemma:distinct_neighbors}
    There exists a permutation matrix $P$ supported on the $0$-cells of $D^{k+1}$ s.t.\ at least one of its cells has $\overline{P}$-neighbors of distinct colors. 
\end{lemma}
\begin{proof}

As mentioned, $G^{k+1}$, the bipartite graph of $0$-cells, is  ($n-2k-1$)-regular. 
Recall that every edge in a
regular bipartite graph can be extended to a perfect matching.

Our plan is first to find a $0$-cell with a blue $\overline{P}$-neighbor in its row and a red $\overline{P}$-neighbor in its column.
Denote by $e$ the edge in $G^{k+1}$ that corresponds to such a cell, and let us find
a perfect matching in $G^{k+1}$ that contains $e$. We will denote the adjacency matrix 
of this matching by $P$.  

It remains to show that such a $0$-cell exists. 
Say wlog, that at least a half of $\overline{P}$'s cells are red. 
As we saw already, $\overline{P}$ is bichromatic, so let
$(x,y')\in \overline{P}$ be a blue cell. We turn to find an index $y$, s.t.\ $(x,y)$ is a $0$-cell and its second $\overline{P}$-neighbor, denoted by $(x',y)$, is red (see right panel of \autoref{fig:stage 3 - P}). 
The number of $0$-cells in row $x$ of $D^{k+1}$ is $n-2k-1 \approx \frac{4}{5} n$.
An index $y$ is ruled out if its corresponding candidate $(x',y)\in \overline{P}$ is blue.
But $\overline{P}$ has at most $\frac{n}{2}$ blue cells, so, there remain at 
least $\frac{4}{5} n - \frac 12 n=\frac{3}{10}n$ acceptable indices $y$.

\end{proof}

\begin{figure}[h!]
    \centering    \includegraphics[width=1.0\linewidth]{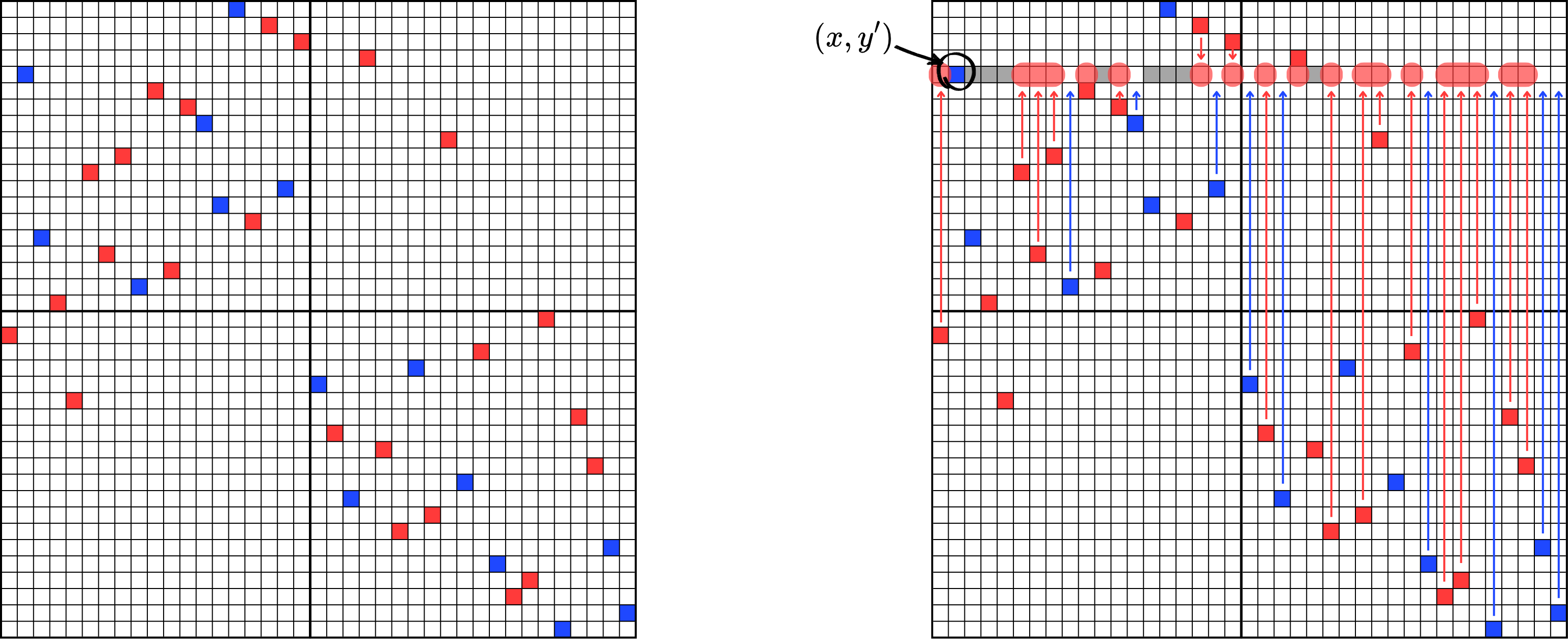}
    \caption{
Left panel: the chosen bi-chromatic perfect matching $\overline{P}$.
Right panel: constructing the perfect matching $P$.
We start from a blue cell $(x,y') \in \overline{P}$, and examine the $0$-cells in row $x$ ($\frac{1}{2}$-cells are shown in gray).
Each such $0$-cell has a $\overline{P}$-neighbor in its column, indicated by a vertical arrow. 
Highlighted in red are the $0$-cells with a red $\overline{P}$-neighbor, which represent the valid candidates for $(x,y)$. 
Our counting argument shows that at least one such cell exists, so that some $0$-cell has neighbors of distinct colors.}
\label{fig:stage 3 - P}
\end{figure}

\begin{figure}[ht]
    \hspace*{2.3cm}
    \includegraphics[width=0.55\linewidth]{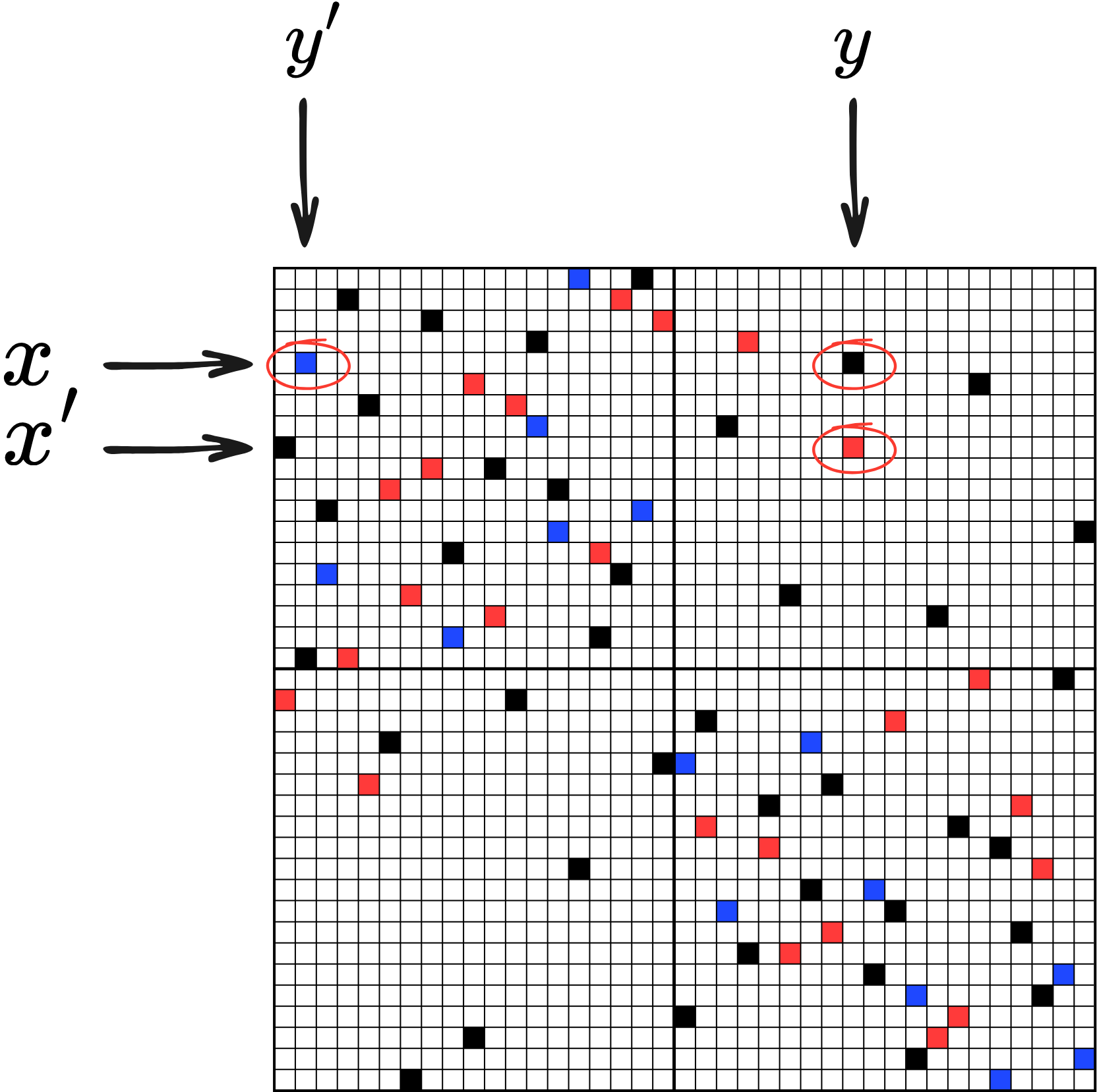}
    \caption{Layer $k+2$: a bi-colored permutation of the $\frac{1}{2}$-cells (red and blue), and another permutation of the $0$-cells (black) s.t.\ $U_{k+2}$ is not $2$-colorable.}
    \label{fig:stage 3 - permutations}
\end{figure}

Finally, we assign the cells that correspond to $\overline{P}$ and $P$:
$$H(a, b, k+2) \xleftarrow{} \frac{1}{2} \text{~for~} (a,b)\in \overline{P}\cup P
\text{~and zeros elsewhere}.$$

We conclude that $U_{k+2}$ is not bipartite, 
as witnessed by the vertex that corresponds to the cell $(x,y,k+2)$, 
which cannot be colored in either blue or red.
Consequently, $U_{i}$ is not bipartite for all $i\ge k+2$.
Additionally, the graph $U_{k+2}$ is connected, since every cell of $\overline{P}$ shares a shaft with a non-zero cell from the previous layers, and every cell of $P$ shares a row or a column with a cell of $\overline{P}$.

Note, and this is crucial, for the subsequent steps in our construction, that the graphs $G^{k+2}$ and $\overline{G}^{k+2}$ remain regular.
This is guaranteed because in each row (resp.\ column) of $D^{k+2}$ the number of $1$-cells is $2$, of $\frac{1}{2}$-cells is $2k$ and of $0$-cells is $n-2k-2$ (see \autoref{fig:stage-3-after}). 

\begin{figure}[h!]
    \centering
    \includegraphics[width=0.42\linewidth]{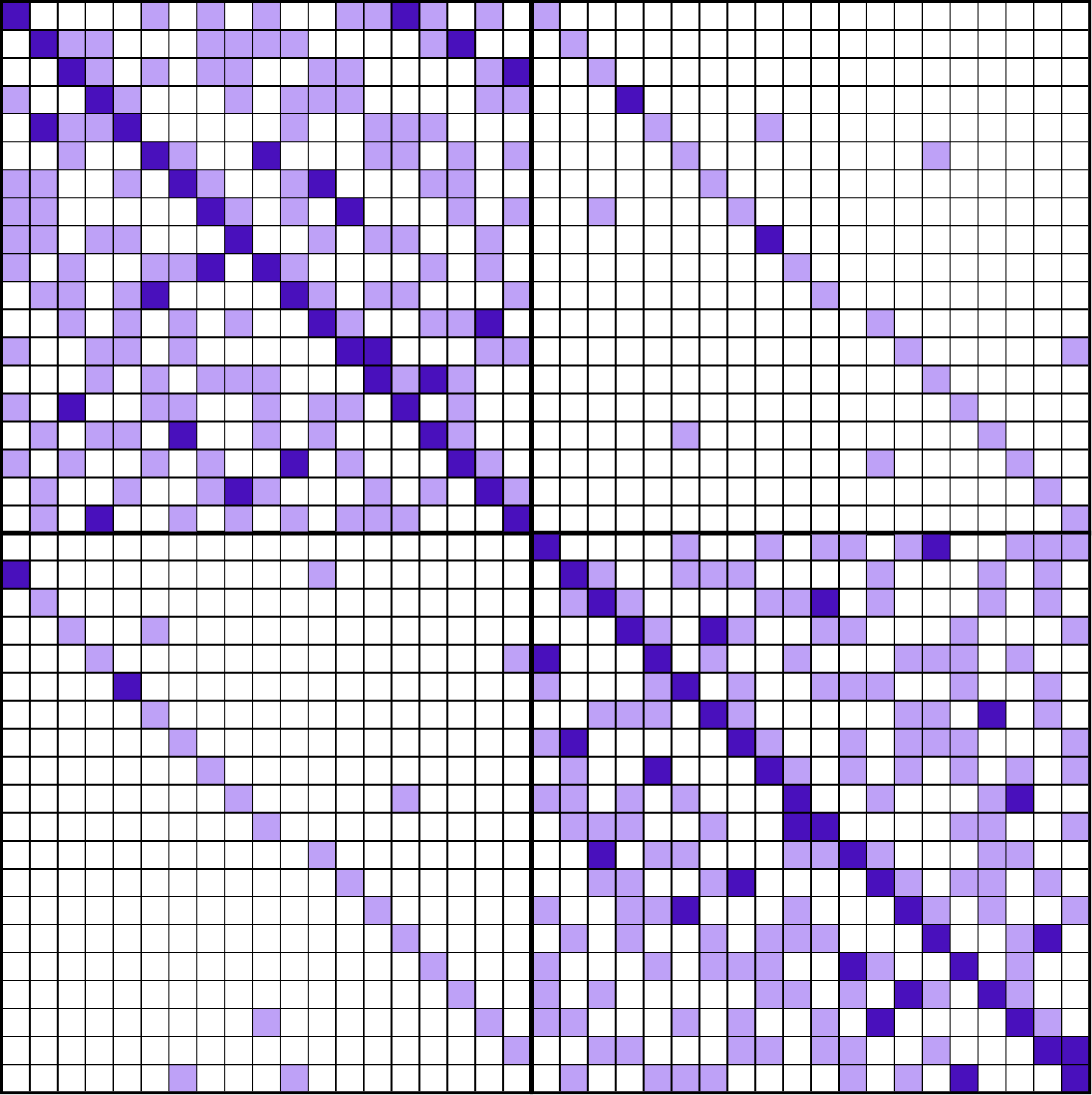}
    \caption{The construction of slice $(1,k+2)$, dark purple represents the $1$-cells, light purple represents the $\frac{1}{2}$-cells of $D^{k+2}(H)$.}
    \label{fig:stage-3-after}
\end{figure}

\subsection{Stage 4 - Eliminating $0$-Cells: Layers $k+3, \dots ,n-k$} \label{stage-4}
In the remaining two stages, we complete the construction of the array $H$, 
until finally $D^n(H)$ is the all-$1$'s matrix. 
To begin, we eliminate all of $D$'s $0$-cells, turning them into 
$\frac{1}{2}$-cells or $1$-cells.
This stage can be viewed as a simplified counterpart to stage $3$. 
The graph $G^{k+2}$ resp.\ $\overline{G}^{k+2}$ are ($n-2k-2$)-regular 
resp.\ ($2k$)-regular, and edge-disjoint. 
Let $M\subseteq E^{k+2}$ resp.\ $\overline{M} \subseteq \overline{E}^{k+2}$ be a perfect matching in $G^{k+2}$ resp.\ $\overline{G}^{k+2}$.
Define layer $(k+3)$ of $H$ by
\begin{equation*}
        H(a, b, k+3) \xleftarrow{} \frac{1}{2} \text{~for~}(a,b)\in M\cup \overline{M}
        \text{~and zeros elsewhere~}
\end{equation*}


Again, this step preserves the regularity of the graphs as we proceed from a
layer to the subsequent one. The number of $0$-cells per row or column in $D$ is reduced by one, while the number of $\frac{1}{2}$-cells is unchanged. 
We repeat this process iteratively until no $0$-cells remain, which
happens after $n-2k-2$ repeats, as we proceed from layer $k+3$ through layer 
$n-k$. 

Also, the graph $U_{i}$ remains connected for each layer $i = k+3, \ldots ,n-k$,
because each cell in layer $i$ either shares a shaft with a non-zero cell from a previous layer or shares a row or a column with a cell that does. 

\subsection{Stage 5 - Finalizing the Construction: Layers $n-k+1,\dots ,n$}\label{stage-5}
Presently $D^{n-k+1}(H)$ no longer has any $0$-cells. In this final stage we change every
remaining $\frac{1}{2}$-cell
into a $1$-cell, thereby completing the construction.
Each layer is constructed by using two disjoint perfect matchings, but unlike the previous 
stage, both matchings are extracted from the same graph. Since
the graph $\overline{G}^{n-k}$ is ($2k$)-regular, it contains
a perfect matching $\overline{M}_1\subseteq \overline{E}^{n-k}$. 
Removing $\overline{M}_1$ leaves a ($2k-1$)-regular residual graph, in
which we find another disjoint perfect matching $\overline{M}_2$. 
Define the $(n-k+1)$-th layer by assigning
$$H(a, b, n-k+1) \xleftarrow{} \frac{1}{2}\text{~for~} (a,b)\in \overline{M}_1\cup \overline{M}_2 \text{~and zeros elsewhere~}$$

We repeat this process, eliminating two $\frac 12$-cells per row or column in each layer, until the construction is complete after $k$ layers. 
As in the previous step, the graph $\overline{G}^{i}$ is
regular and $U_{i}$ is connected for all $i = n-k+1, \dots ,n$. 

\section{The main theorem}
We are now ready for our final calculation. We prove:

\begin{theorem} \label{thm:main}
The polytope $\Delta_n$ has at least $n^{(2-o(1))\cdot n^2}$ vertices.
\end{theorem}
\begin{proof}

The total number of arrays $H\in\mathcal{H}$ that our construction
produces is the product of the number of choices available at the various layers.
We split our calculation into three, paralleling the above stages.
We have set $k:\approx\frac{n}{10}$. In what follows, we
work with $k=\frac{n}{10}$, and note that any small deviation from this value
has essentially no effect on the calculations that we carry out.
We first state the estimates, with proofs appearing subsequently:
\begin{itemize}
\item In {\it stage $1$, step $2$}, concerning layers $1,\dots,k$, with two Hamiltonian cycles per layer, one in $P$ and the other in $Q$, the number of possible choices is $n^{(\frac{1}{5}-o(1))\cdot n^2}$.
\item In {\em stage $4$}, concerning layers $k+3,\dots,n-k$, with two matchings per layer, one in $G^i$ and the other in $\overline{G}^i$, the number of possible choices is $n^{(\frac{8}{5}-o(1))\cdot n^2}$.
\item In {\em stage $5$}, layers $n-k+1,\dots,n$, with two matchings per layer, both in $\overline{G}^i$, the number of possible choices is $n^{(\frac{1}{5}-o(1))\cdot n^2}$.
\end{itemize} 
These bounds yield the stated result, and we turn to provide 
proofs and explanations for each stage separately. 

\subsection{The contribution of stage $1$ (step $2$)}
For each layer $1\leq i\leq k$ we constructed, 
in stage $1$ (step $2$), a Hamiltonian cycle in the regular
bipartite graph $G^i_P$ (eq. $G^i_Q$). As shown by
Kahn and Cuckler \cite{cuckler2009hamiltonian}, a balanced
bipartite  
graph with a total of $|V(G)|$ vertices and minimal degree
$d\geq \frac{|V(G)|}{4}$ has $\Psi(G)\geq\left(\frac{d}{e+o(1)}\right)^{|V(G)|}$ Hamiltonian cycles. 
In our case, as described in \hyperref[section-2-stage-1]{step 2},
the graph $G^i_P$ has $2(\big\lfloor \frac{n}{2} \big\rfloor-c_i)$ vertices,
where $c_i \le 5$. Its vertex degrees are at least
$\big\lfloor \frac{n}{2} \big\rfloor-c_i-5-2i+2$. Therefore, in constructing layer
$1\leq i\leq k$, the number of Hamiltonian cycles
from which to choose is at least 

$$\Psi(G^i_P) \cdot \Psi(G^i_Q)\geq
\left(\left(\frac{\big\lfloor \frac{n}{2} \big\rfloor-c_i-5-2i+2}{e+o(1)}\right)^{2\left(\lfloor \frac{n}{2} \rfloor-c_i\right)} \right)^2\geq 
\left(\frac{\big\lfloor \frac{n}{2} \big\rfloor-2i-8}{e+o(1)}\right)^{4\left(\lfloor \frac{n}{2} \rfloor-5\right)}.$$
Going over all layers $i=1,\dots, k$,
the total number of choices of Hamiltonian cycles is at least 

$$\prod_{i=1}^k \left(\Psi(G^i_P) \cdot \Psi(G^i_Q)\right) 
\geq \prod_{i=1}^k \left(\frac{\big\lfloor \frac{n}{2} \big\rfloor-2i-8}{e+o(1)}\right)^{4\left(\lfloor \frac{n}{2} \rfloor-5\right)}\geq
\left(\frac{\frac{3n}{10}-8}{e+o(1)}\right)^{
\frac{n^2}{5}-2n} \ge
n^{(\frac{1}{5}-o(1))\cdot n^2}$$

\subsection{The contribution of stage $4$}
In \hyperref[stage-4]{stage 4}, for each layer $i=k+3,\dots , n-k$ we select two perfect matchings, one in $G^i$ and one in $\overline{G}^i$. 
Here $(x,y)$ is an edge in $G^i$, resp.\ $\overline{G}^i$
if $D^i(x,y)=0$, resp.\ $D^i(x,y)=\frac{1}{2}$. By applying
claim \ref{claim:bound_d-regular} to the adjacency matrices of these graphs,
we obtain the following lower bound on the count of
possible matchings in the $i$-th layer:

$$\text{per}\left(M^i\right) \geq \left( \frac{n-k+1-i}{e}\right)^n \ ;
\ \text{per}\left(\overline{M}^i\right)  \geq \left( \frac{2k}{e}\right)^n .$$
In total, this stage provides us with the following lower bound:

$$\prod_{i=k+3}^{n-k} \left(\left( \frac{n-k+1-i}{e}\right) \cdot \left( \frac{2k}{e}\right)\right)^n=
\left( \frac{(n-2k-2)!}{e^{n-2k-2}} \cdot\left( \frac{2k}{e}\right)^{n-2k-2}\right)^n.$$
By Stirling's formula, this equals
$$\left(\left(\left(1+o(1)\right)\cdot\frac{(n-2k-2)}{e^2}\right)^{n-2k-2}\cdot\left(\frac{2k}{e}\right)^{n-2k-2}\right)^n=$$
$$
\left((1+o(1))\cdot \frac{(n-2k-2)\cdot 2k}{e^3}\right)^{n\cdot (n-2k-2)}=
n^{(\frac{8}{5}-o(1))\cdot n^2}$$

\subsection{The contribution of stage $5$}
In the \hyperref[stage-5]{final stage}, we select
at each layer $i= n-k+1, \dots , n$, two disjoint perfect matchings in $\overline{G}^i$. 
Again, applying claim \ref{claim:bound_d-regular} to the adjacency matrices of the graphs $\overline{G}^i$ and $\overline{G}^i\setminus \overline{M}^i_1$ and invoking Stirling's formula yields the following bound for this stage.
Note that this process is equivalent to iteratively selecting a perfect matching, and proceeding with the residual graph, repeated $2k$ times.

$$\prod_{i=n-k+1}^{n}\left( \text{per}\left(\overline{M}^i_1\right) \cdot \text{per}\left(\overline{M}^i_2\right) \right) \geq 
\left( \frac{(2k)!}{e^{2k}}\right)^n =
\left( \left((1+o(1))\cdot\frac{2k}{e^2}\right)^{2k}\right)^n = 
n^{(\frac{1}{5}-o(1))\cdot n^2}$$

\end{proof}

\section{So, how many vertices does $\Delta_n$ have?}

This question remains presently open. We can show, however,

\begin{prop}\label[proposition]{prop:ub}
$\Delta_n$ has no more than $n^{(3+o(1))\cdot n^2}$ vertices.
\end{prop}

To prove this proposition we need to consider 
the support size of vertices in $\Delta_n$. Every 
vertex $Z$ of $\Delta_n$ clearly has support size at least $n^2$,
since every row of the array $Z$ must contain a non-zero entry. The polytope
$\Delta_n$ is $(n-1)^3$-dimensional. It
is defined by $n^3$ inequalities which are categorized as follows:
\begin{itemize}
\item
$(n-1)^3$ non-negativity inequalities one per each
entry $Z(i,j,k)$ with $1\le i, j, k \le n-1$.
\item 
$3(n-1)^2$ upper bounds on the sum of every line in $Z$.
\item 
$3(n-1)$ lower bounds on the sum of each wall or layer of $Z$.
\item 
One upper bound on $\sum_{i,j,k=1}^{n-1} Z(i,j,k)$.
\end{itemize}
In a vertex of $\Delta_n$ at least
$(n-1)^3$ inequalities must hold with equality.
So, at least 
$3(n-1)^2+3(n-1)+1$ entries of $Z$ must be positive,
yielding a lower bound on the support size of each vertex of $\Delta_n$.
The proof of \cref{prop:ub} follows from:

\begin{prop}\label[proposition]{prop:uniqueness}
Every vertex of $\Delta_n$ is uniquely determined by its support.
\end{prop}
\begin{proof}
Let $Z \in \Delta_n$ be a vertex of $\Delta_n$. Suppose towards contradiction
that some $W \neq Z$ belongs to $\Delta_n$, and yet 
$\text{supp}(W) = \text{supp}(Z)$. Denote $F=Z-W$. Clearly $F$ is not identically
zero and $\text{supp}(F) \subseteq \text{supp}(W) = \text{supp}(Z)$.
It follows that for a sufficiently small $\epsilon > 0$, both $(Z + \epsilon \cdot F)$ and $(Z - \epsilon\cdot F)$ are two distinct members of $\Delta_n$. But then
$Z$ is the convex combination 
$Z = \frac{1}{2}(Z + \epsilon\cdot F) + \frac{1}{2}(Z - \epsilon\cdot F)$, 
contradicting the assumption that $Z$ is a vertex.
\end{proof}

\subsection{Some numerics}
It is easy to sample vertices of $\Delta_n$, though uniform sampling still eludes us. We have conducted some preliminary simulations to this end.
We write the defining inequalities of $\Delta_n$ and use an LP-solver
to maximize a random objective function on this polytope.
This sampling of vertices of $\Delta_n$ is most likely not uniform.
Still, it is worth noting that for substantial $n$, individual lines have varying 
support sizes, though support sizes of $2$ and $3$ are
frequently observed. Also,
as $n$ increases, the non-zero values appearing in the vertices become increasingly diverse.

This leads to several questions:
\begin{itemize}
\item 
How large can the support size be in a single line of a vertex of $\Delta_n$? 
\item
Is there any hope to say anything about higher-dimensional faces of $\Delta_n$?
We recall that in contrast, the edges of the Birkhoff Polytope are well characterized.
\end{itemize}

\begin{figure}[p]
    \centering
    \hspace*{-1.3cm} 
    \begin{tabular}{cccccc}
        & \textbf{Sample 1} & \textbf{Sample 2} & \textbf{Sample 3} & \textbf{Sample 4} & \textbf{Sample 5} \\
        
        \rotatebox{90}{\textbf{$n=10$}} & 
        \includegraphics[width=0.19\textwidth]{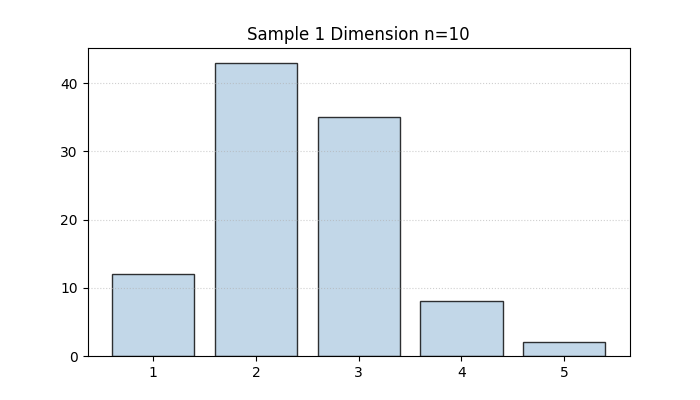} & 
        \includegraphics[width=0.19\textwidth]{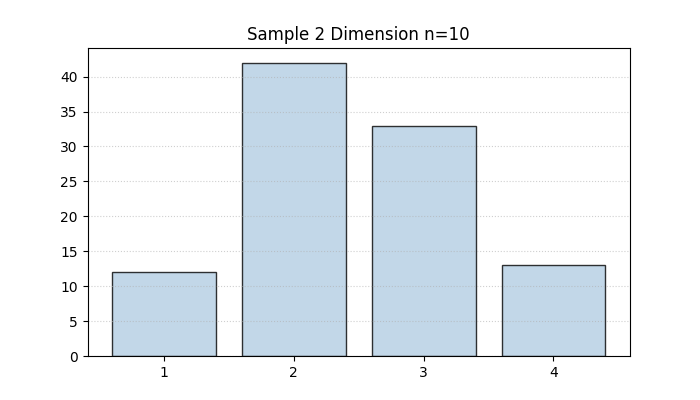} & 
        \includegraphics[width=0.19\textwidth]{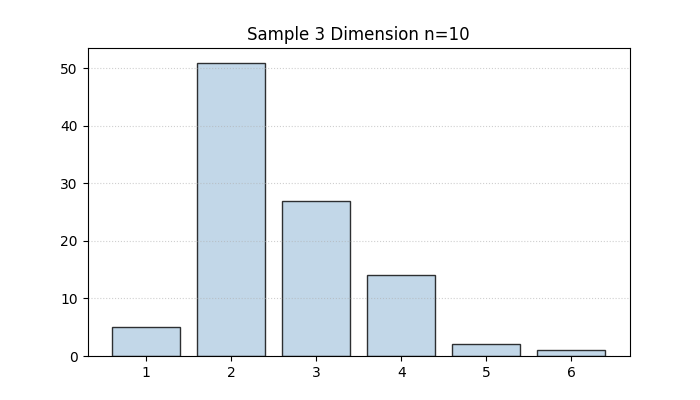} & 
        \includegraphics[width=0.19\textwidth]{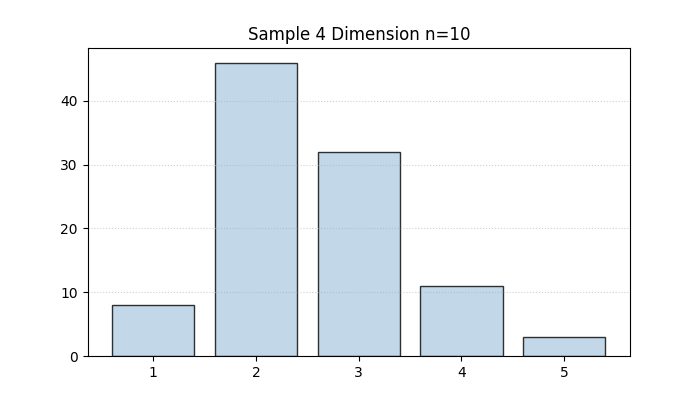} & 
        \includegraphics[width=0.19\textwidth]{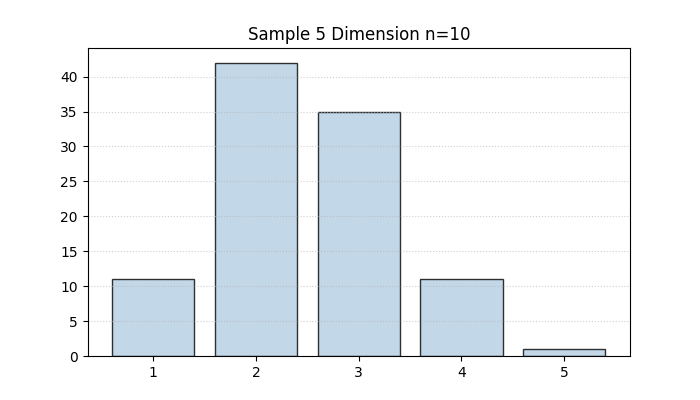} \\[1ex]

        \rotatebox{90}{\textbf{$n=20$}} & 
        \includegraphics[width=0.19\textwidth]{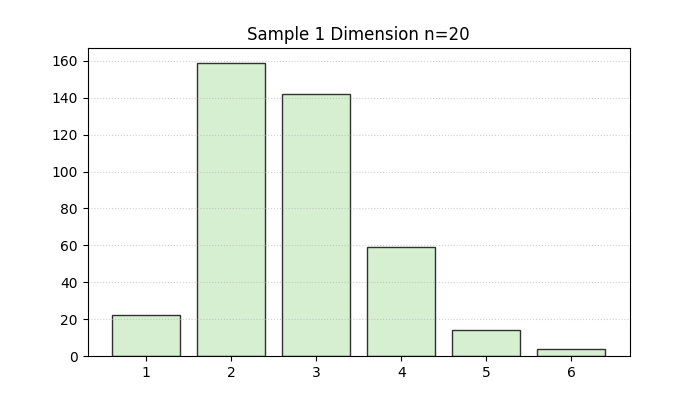} & 
        \includegraphics[width=0.19\textwidth]{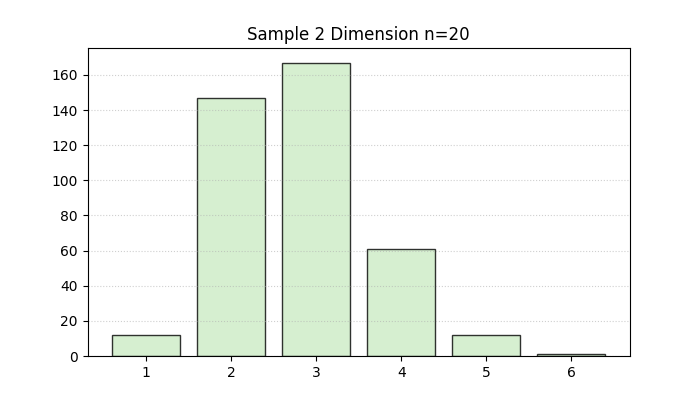} & 
        \includegraphics[width=0.19\textwidth]{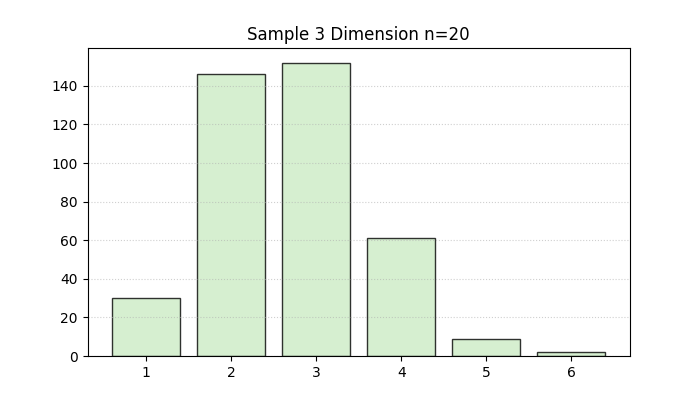} & 
        \includegraphics[width=0.19\textwidth]{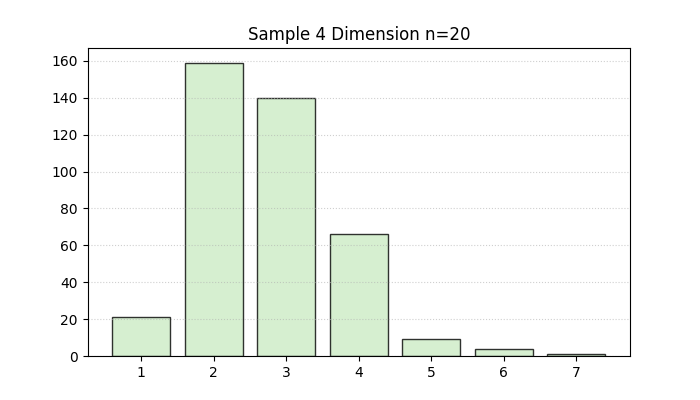} & 
        \includegraphics[width=0.19\textwidth]{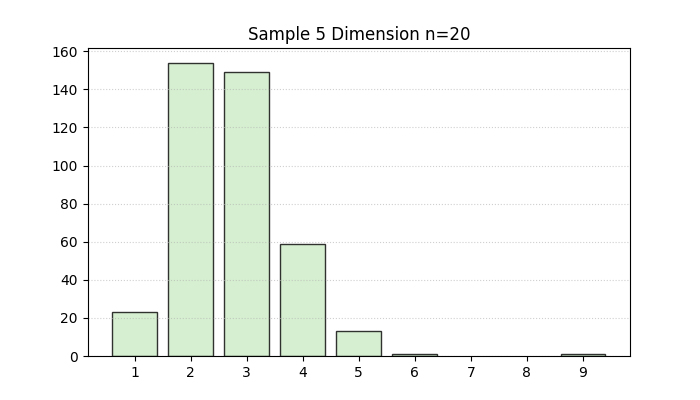} \\[1ex]

        \rotatebox{90}{\textbf{$n=30$}} & 
        \includegraphics[width=0.19\textwidth]{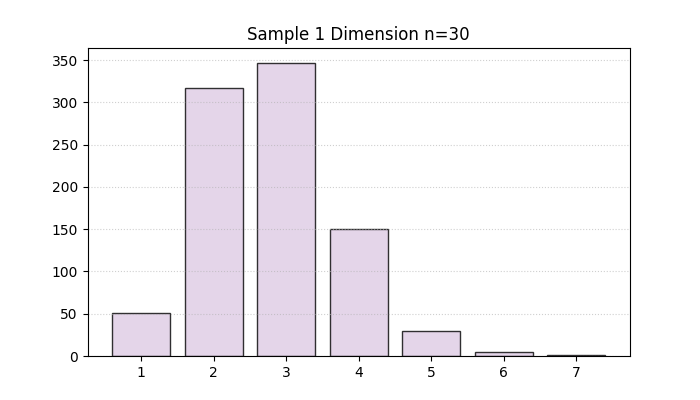} & 
        \includegraphics[width=0.19\textwidth]{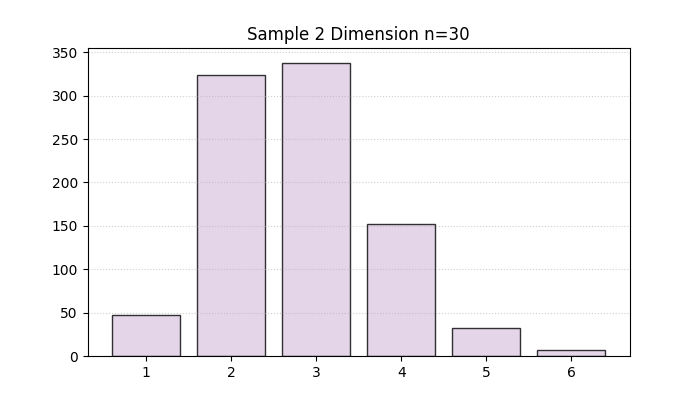} & 
        \includegraphics[width=0.19\textwidth]{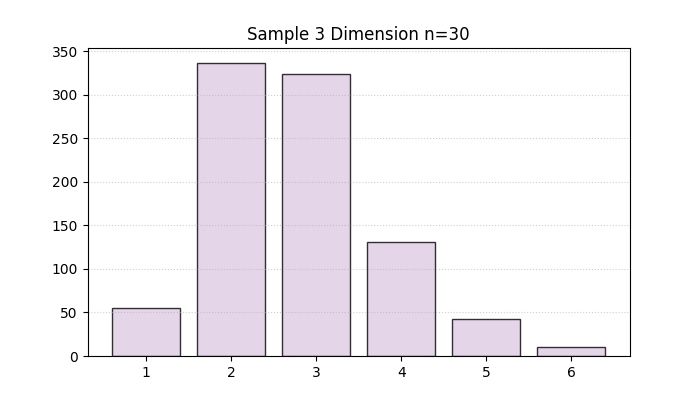} & 
        \includegraphics[width=0.19\textwidth]{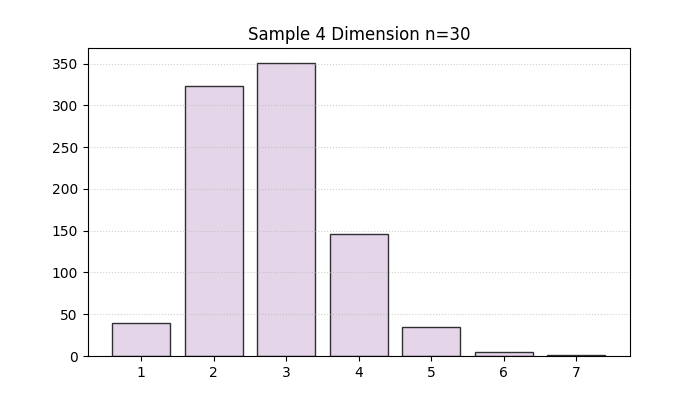} & 
        \includegraphics[width=0.19\textwidth]{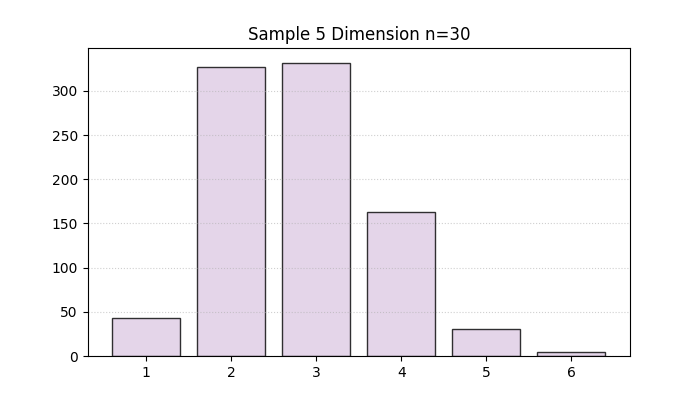} \\[1ex]

        \rotatebox{90}{\textbf{$n=40$}} & 
        \includegraphics[width=0.19\textwidth]{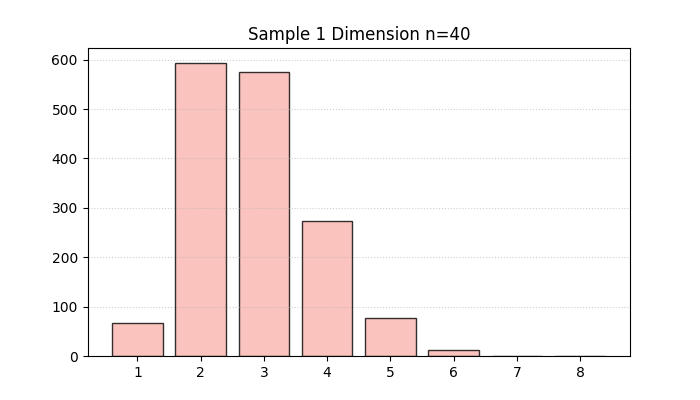} & 
        \includegraphics[width=0.19\textwidth]{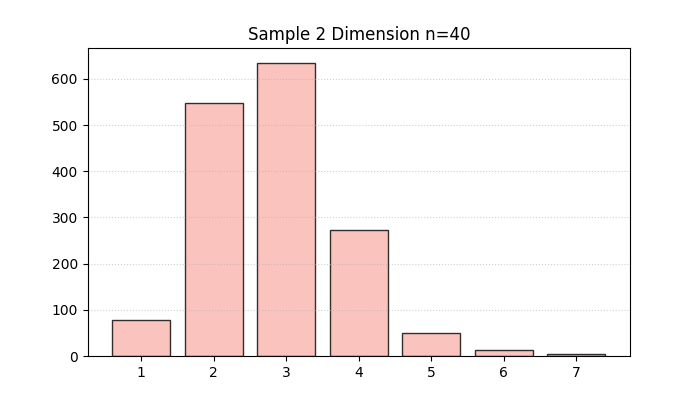} & 
        \includegraphics[width=0.19\textwidth]{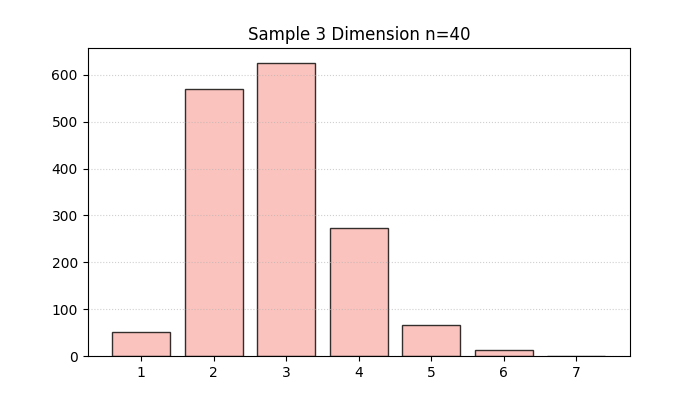} & 
        \includegraphics[width=0.19\textwidth]{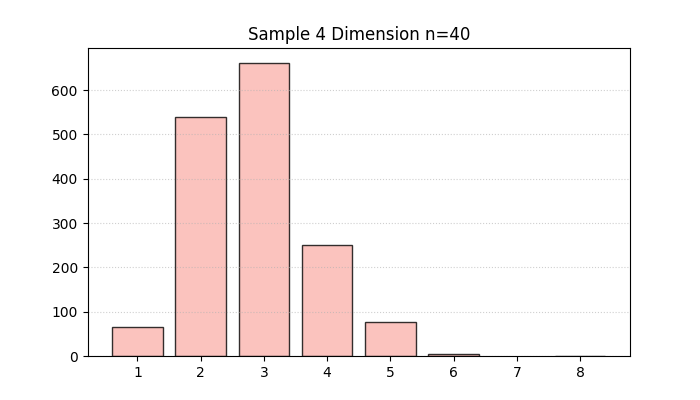} & 
        \includegraphics[width=0.19\textwidth]{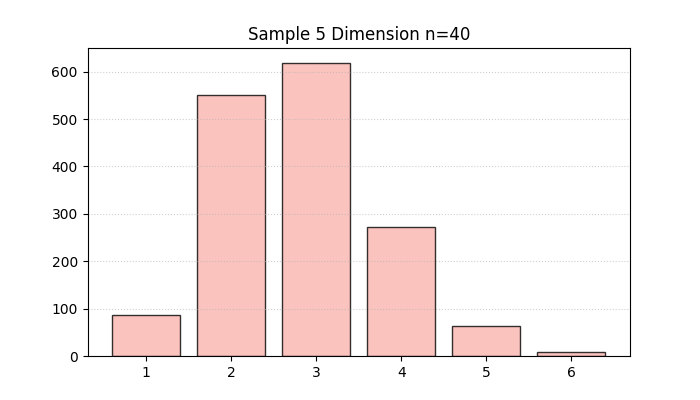} \\[1ex]

        \rotatebox{90}{\textbf{$n=50$}} & 
        \includegraphics[width=0.19\textwidth]{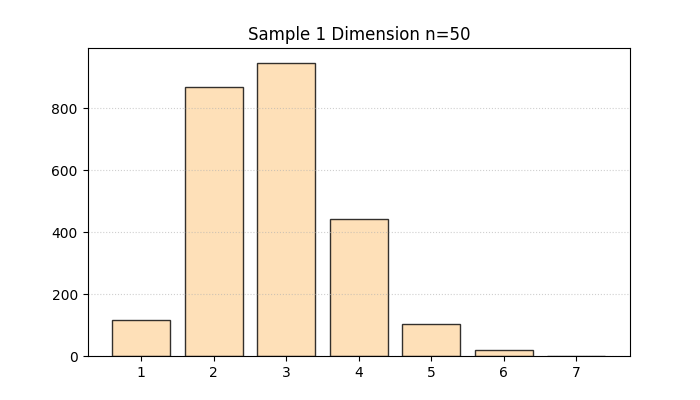} & 
        \includegraphics[width=0.19\textwidth]{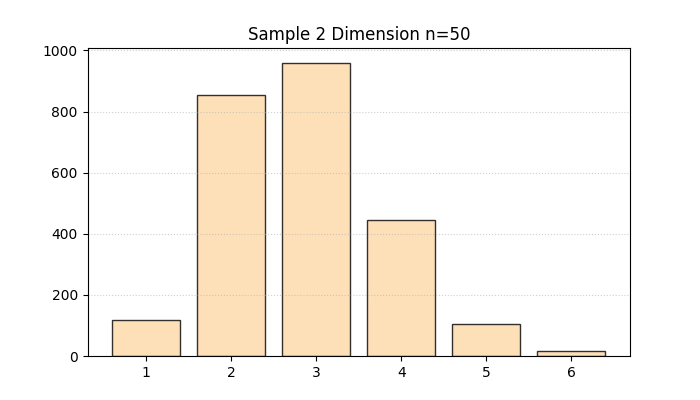} & 
        \includegraphics[width=0.19\textwidth]{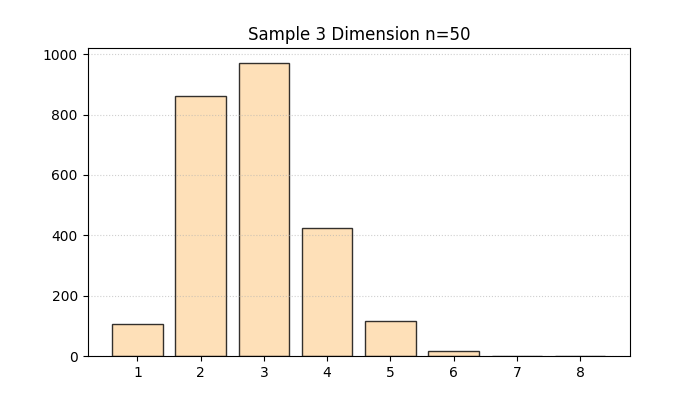} & 
        \includegraphics[width=0.19\textwidth]{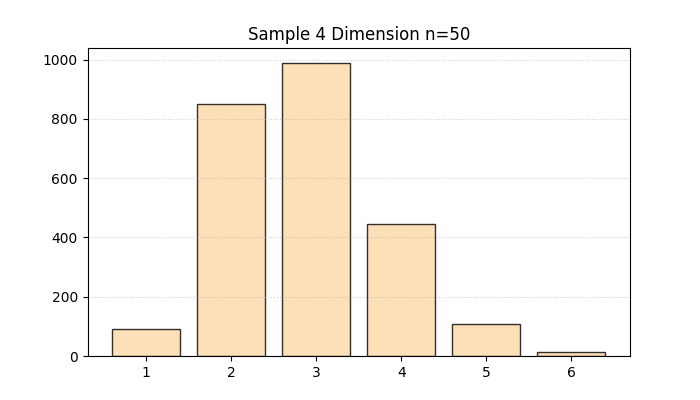} & 
        \includegraphics[width=0.19\textwidth]{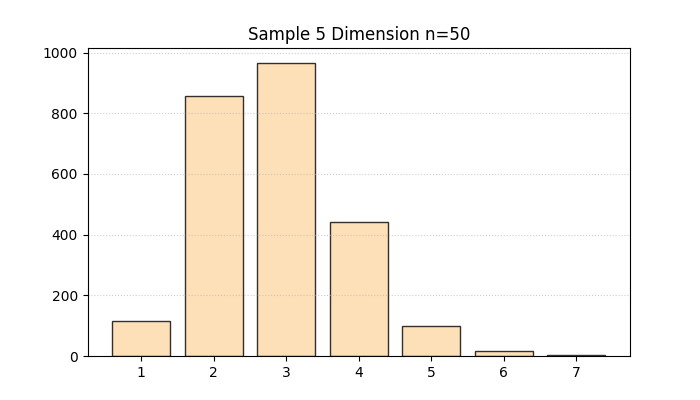} \\[1ex]

        \rotatebox{90}{\textbf{$n=60$}} & 
        \includegraphics[width=0.19\textwidth]{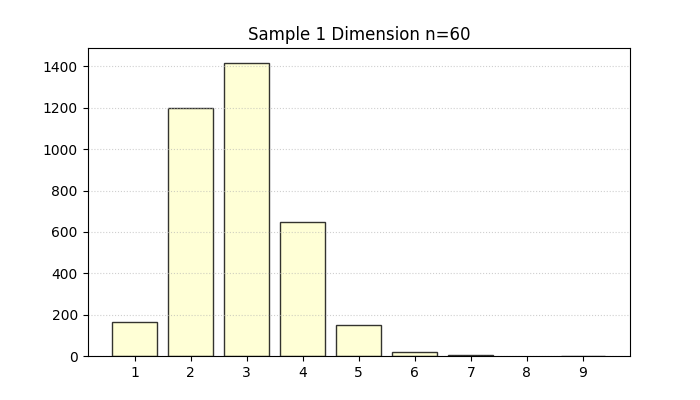} & 
        \includegraphics[width=0.19\textwidth]{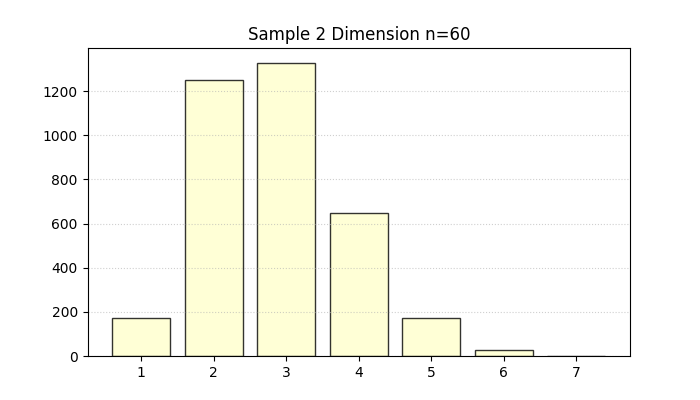} & 
        \includegraphics[width=0.19\textwidth]{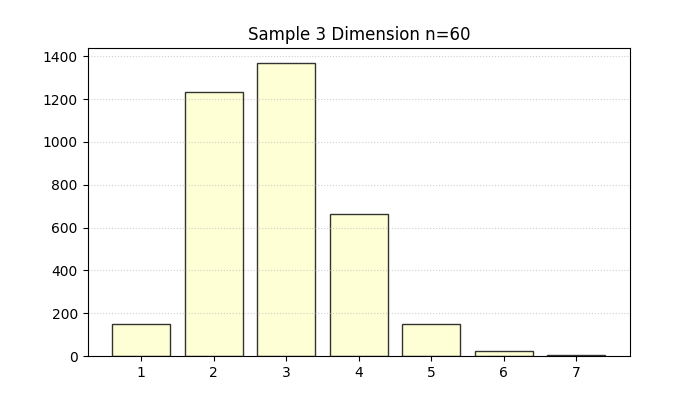} & 
        \includegraphics[width=0.19\textwidth]{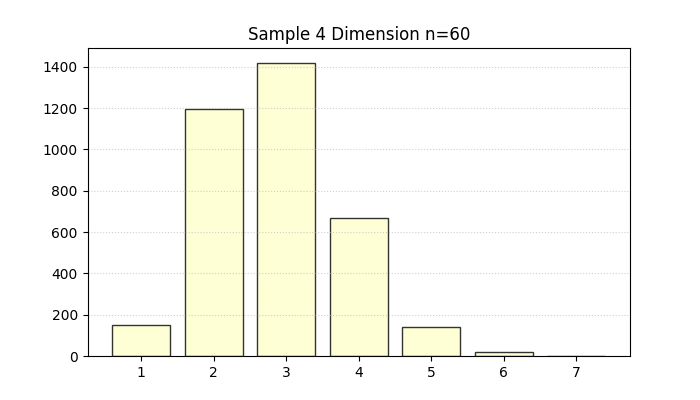} & 
        \includegraphics[width=0.19\textwidth]{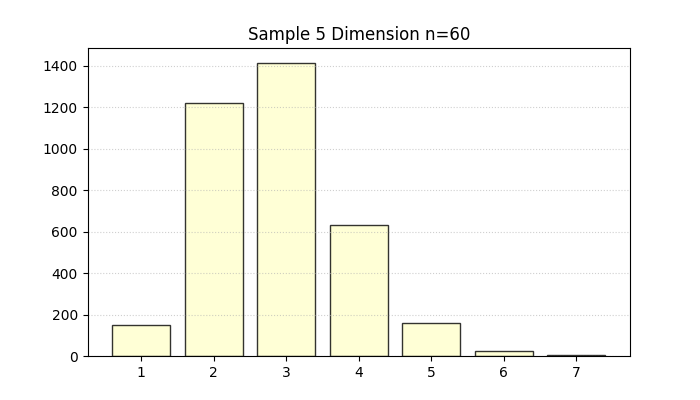} \\
    \end{tabular}
    \caption{Histograms of row support sizes for randomly sampled tri-stochastic vertices across dimensions $n=10$ through $n=60$.}
    \label{fig:expanded_histogram_grid}
\end{figure}

\clearpage
\bibliographystyle{plain}
\bibliography{references}

@article{moon1963hamiltonian,
  title={{On Hamiltonian bipartite graphs}},
  author={Moon, John and Moser, Leo},
  journal={Israel Journal of Mathematics},
  volume={1},
  number={3},
  pages={163--165},
  year={1963},
  publisher={Springer-Verlag New York}
}

@article{dirac1952some,
  title={{Some theorems on abstract graphs}},
  author={Dirac, Gabriel Andrew},
  journal={Proceedings of the London Mathematical Society},
  volume={3},
  number={1},
  pages={69--81},
  year={1952},
  publisher={Oxford University Press}
}

@article{bregman1973some,
  title={{Some properties of nonnegative matrices and their permanents}},
  journal={Soviet Mathematics Doklady},
  author={Br{\`e}gman, Lev Meerovich},
  volume={14},
  pages={945--949},
  year={1973}
}

@article{minc1963upper,
  title={{Upper bounds for permanents of (0,1)-matrices}},
  journal={Bulletin of the American Mathematical Society},
  author={Minc, Henryk},
  volume={69},
  pages={789--791},
  year={1963}
}

@article{hummel1940note,
  title={{A note on Stirling's formula}},
  author={Hummel, P.M.},
  journal={The American Mathematical Monthly},
  volume={47},
  number={2},
  pages={97--99},
  year={1940},
  publisher={JSTOR}
}

@article{egorychev1981solution,
  title={{The solution of Van der Waerden's problem for permanents}},
  author={Egorychev, Gregory P.},
  journal={Advances in Mathematics},
  volume={42},
  number={3},
  pages={299--305},
  year={1981},
  publisher={Academic Press}
}

@article{falikman1981proof,
  title={{Proof of the Van der Waerden conjecture regarding the permanent of a doubly stochastic matrix}},
  author={Falikman, Dmitry I.},
  journal={Mathematical notes of the Academy of Sciences of the USSR},
  volume={29},
  number={6},
  pages={475--479},
  year={1981},
  publisher={Springer}
}

@article{cuckler2009hamiltonian,
  title={{Hamiltonian cycles in Dirac graphs}},
  author={Cuckler, Bill and Kahn, Jeff},
  journal={Combinatorica},
  volume={29},
  number={3},
  pages={299--326},
  year={2009},
  publisher={Springer}
}

@article{birkhoff1946tres,
  title={{Tres observaciones sobre el algebra lineal}},
  author={Birkhoff, Garrett},
  journal={Univ. Nac. Tucuman, Ser. A},
  volume={5},
  pages={147--154},
  year={1946}
}

@article{linial2014vertices,
  title={{On the vertices of the d-dimensional Birkhoff polytope}},
  author={Linial, Nathan and Luria, Zur},
  journal={Discrete \& Computational Geometry},
  volume={51},
  number={1},
  pages={161--170},
  year={2014},
  publisher={Springer}
}

@article{luria_count_hd_prmt,
  author       = {Nathan Linial and
                  Zur Luria},
  title        = {{An upper bound on the number of high-dimensional permutations}},
  journal      = {Comb.},
  volume       = {34},
  number       = {4},
  pages        = {471--486},
  year         = {2014},
  url          = {https://doi.org/10.1007/s00493-011-2842-8},
  doi          = {10.1007/S00493-011-2842-8},
  bibsource    = {dblp computer science bibliography, https://dblp.org}
}

@book{van_lint_wilson,
  title={{A course in combinatorics}},
  author={Van Lint, Jacobus Hendricus and Wilson, Richard Michael},
  year={2001},
  publisher={Cambridge university press}
}
\end{document}